\documentclass[a4paper, 9pt]{article}
\usepackage{graphics}
\usepackage[dvips]{color}

\usepackage{lscape}
\usepackage{latexsym}
\usepackage{amsmath,verbatim}
\usepackage{graphics}
\usepackage{amsthm}
\usepackage{amssymb}
\usepackage{makeidx}
\usepackage{stmaryrd}
\usepackage{multicol}
\usepackage{helvet}
\usepackage{bm}
\usepackage[all]{xy}
\newfont{\Fr}{eufm10}
\newfont{\Sc}{eusm10}
\newfont{\Bb}{msbm10}
\newfont{\Am}{msam10}
\newfont{\am}{msam7}
\numberwithin{equation}{section}
\newtheorem{theorem}{Theorem}[section]
\newtheorem{proposition}[theorem]{Proposition}
\newtheorem{lemma}[theorem]{Lemma}
\newtheorem{corollary}[theorem]{Corollary}
\newtheorem{claim}{Claim}{\bf}{\it}

\newtheorem{ftheorem}{Theorem}{\bf}{\it}
{\bf}{\it}
\newtheorem{fcorollary}[ftheorem]{Corollary}{\bf}{\it}
\theoremstyle{definition}
\newtheorem{definition}[theorem]{Definition}

\newtheorem{fdefinition}[ftheorem]{Definition}{\bf}{\rm}
\newtheorem{fact}[theorem]{Fact}

\theoremstyle{remark}

\newtheorem{remark}[theorem]{Remark}
\newtheorem{fremark}[ftheorem]{Remark}{\bf}{\it}

\newtheorem{definition and corollary}[theorem]{Definition and Corollary}

\newtheorem{problem}[theorem]{Problem}

\newtheorem{fexample}[ftheorem]{Example}{\it}{\rm}

\newcommand{\Hom}{\mbox{\rm Hom}}

\newcommand{\MID}{\! \! \mid}

\newcommand{\h}{\mathfrac{\h}}

\title{A homological study of Green polynomials\footnote{The word ``green'' means `midori' in Japanese.}\, \footnote{French translation: Une \'etude homologique des polyn\^omes de Green}
}
\author{Syu \textsc{Kato} \footnote{Department of Mathematics, Kyoto University, Oiwake Kita-Shirakawa Sakyo Kyoto 606-8502, Japan. \tt{E-mail:syuchan@math.kyoto-u.ac.jp}} \footnote{Research supported in part by Max-Planck Institute f\"ur Mathematik in Bonn and JSPS Grant-in-Aid for Young Scientists (B) 23-740014.}}

\begin{document}
\maketitle

\begin{abstract}
We interpret the orthogonality relation of Kostka polynomials arising from complex reflection groups ([Shoji, Invent. Math. 74 (1983), J. Algebra 245 (2001)] and [Lusztig, Adv. Math. 61 (1986)]) in terms of homological algebra. This leads us to the notion of Kostka system, which can be seen as a categorical counterpart of Kostka polynomials. Then, we show that every generalized Springer correspondence ([Lusztig, Invent. Math. 75 (1984)]) in a good characteristic gives rise to a Kostka system. This enables us to see the top-term generation property of the (twisted) homology of generalized Springer fibers, and the transition formula of Kostka polynomials between two generalized Springer correspondences of type $\mathsf{BC}$. The latter provides an inductive algorithm to compute Kostka polynomials by upgrading [Ciubotaru-Kato-K, Invent. Math. 178 (2012)] \S 3 to its graded version. In the appendices, we present purely algebraic proofs that Kostka systems exist for type $\mathsf{A}$ and asymptotic type $\mathsf{BC}$ cases, and therefore one can skip geometric sections \S 3--5 to see the key ideas and basic examples/techniques.
\end{abstract}

\section*{Introduction}
Green polynomials attached to a reductive group is a family of polynomials indexed by two conjugacy classes of their (endoscopic) Weyl groups, depending on a variable $t$ roughly represents the cardinality of the base field. Introduced by Green \cite{Gr} for $\mathop{GL} ( n, \mathbb F_q )$ and Deligne-Lusztig \cite{DL} in general, they play a central role in the representation theory of finite groups of Lie types, affine Hecke algebras, $p$-adic groups, and so on. Equivalent to Green polynomials are Kostka polynomials attached to a reductive group, which are $t$-analogues of Kostka numbers in the case of $\mathop{GL} ( n )$. Hence, they appear almost everywhere in representation theory attached to root data.

Despite their natural appearance, not much is known about Kostka polynomials except for type $\mathsf A$. One major reason seems to be the fact that the set of Kostka polynomials admits integral parameters, which actually yield different collections of polynomials even if they arise from character sheaves of finite Chevalley groups (\cite{L-IC,Lu2,Lu3}). In such representation theoretic situation, Lusztig \cite{L-IC} introduced the notion of symbols, which govern the combinatorial data to determine Kostka polynomials by means of their {\it orthogonality relation} (\cite{Sh1,Lu2}). It is generalized by Malle \cite{Ma} and Shoji \cite{Sh2, Sh3} to include the case of complex reflection groups, in which the orthogonality relation is employed as their definition.

Kostka polynomials also appear in the context of elliptic representation theory (\cite{A}), that is the ``cuspidal quotient" of (usual) representation theory. In particular, the study of formal degrees of affine Hecke algebras/$p$-adic groups (\cite{Re,Op,OS,CKK,CT}) revealed the transition pattern of Kostka polynomials evaluated at $t = 1$. This supplies connections among representation theories of infinitely many $p$-adic groups (of different types).

The goal of the present paper is two-fold: One is to afford an algebraic framework of the study of Kostka polynomials of complex reflection groups. The other is to exhibit how the classical results on Kostka polynomials of Weyl groups and the above transition pattern unveil their finer versions in our framework. From these, we expect that our framework is suited to study global structures of families of (the sets of) Kostka polynomials, and to study their connections with elliptic/usual representation theory of reductive groups or ``spetses" (\cite{BMM2}).

For more detailed explanation, we need notations: Let $W$ be a complex reflection group, and let $\mathsf{Irr} \, W$ denote the set of isomorphism classes of irreducible $W$-modules. For each $\chi \in \mathsf{Irr} \, W$, we denote by $\chi^{\vee}$ its dual representation. Let $\mathfrak h$ be a reflection representation of $W$. Form a graded algebra $A _W := \mathbb C W \ltimes \mathbb C [\mathfrak h^*]$ with $\deg w = 0$ ($w \in W$) and $\deg x = 2$ ($x \in \mathfrak h$). Let $A_W \mathchar`-\mathsf{gmod}$ be the category of finitely generated $\mathbb Z$-graded $A_{W}$-modules. For $E, F \in A_W \mathchar`-\mathsf{gmod}$, we define
$$\left< E, F \right> _{\mathsf{gEP}} := \sum _{i \ge 0} (-1)^i \mathsf{gdim} \, \mathrm{ext} ^i _{A_W} (E,F) \in \mathbb Z (\!(t^{1/2})\!),$$
where $\mathrm{ext}$ means the graded extension, and $\mathsf{gdim}$ means the graded dimension (which sends a $\mathbb Z$-graded vector space $V = \oplus _{j \gg - \infty} V_j$ to $\sum _j t^{j/2} \dim V_j$). For each $\chi \in \mathsf{Irr} \, W$, we denote by $L _{\chi}$ the irreducible graded $A_{W}$-module sitting at degree $0$ that is isomorphic to $\chi$ as a $W$-module.

\begin{fdefinition}[$\doteq$ Definition \ref{kos}]\label{fkos}
Let $<$ be a total pre-order on $\mathsf{Irr} \, W$. Then, a Kostka system $\{ K_{\chi} ^{\pm} \} _{\chi} \subset A_W \mathchar`-\mathsf{gmod}$ is a collection such that
\begin{enumerate}
\item Each $K _{\chi} ^{\pm}$ is an indecomposable $A_W$-module with simple head $L_{\chi}$;
\item For each $\chi,\eta \in \mathsf{Irr} \, W$, we have equalities
\begin{align*}
[K_{\chi} ^+] & = [L_{\chi}] + \sum _{\eta > \chi} K _{\chi,\eta} ^+ [L _{\eta}] \hskip 2mm \text{ with } \hskip 2mm K _{\chi,\eta} ^+ \in t \mathbb N [t] \text{ and }\\
[K_{\chi ^{\vee}} ^-] & = [L_{\chi ^{\vee}}] + \sum _{\eta > \chi} K _{\chi,\eta} ^- [L _{\eta ^{\vee}}] \hskip 2mm \text{ with } \hskip 2mm K _{\chi,\eta} ^- \in t \mathbb N [t]
\end{align*}
in the Grothendieck group of $A_W \mathchar`-\mathsf{gmod}$;
\item  We have $\left< K _{\chi} ^+, ( K _{\eta} ^- ) ^* \right> _{\mathsf{gEP}} = 0$ for $\chi ^{\vee} \not\sim \eta$, where $( K _{\eta}^- ) ^*$ is the graded dual of $K _{\eta} ^-$.
\end{enumerate}
If $W$ is a real reflection group, then we have $K^+ _{\chi} = K ^- _{\chi}$ by (the genuine) definition, and we denote them by $K_{\chi}$.
\end{fdefinition}

This definition is slightly weaker than the one presented in the main body of the paper (for simplicity). For Weyl groups, the classical preorders on $\mathsf{Irr} \, W$ reflect the geometry of nilpotent cones and the Springer correspondences.

\begin{ftheorem}[$=$ Theorem \ref{reint}]\label{freint}
For a Kostka system $\{ K_{\chi} ^{\pm} \} _{\chi}$, its graded character multiplicities $K _{\chi,\eta} ^{\pm}$ satisfy the orthogonality relation of Kostka polynomials in the sense of $\cite{Sh1,Lu2,Sh2}$. In particular, a Kostka system is an enhancement of Kostka polynomials.
\end{ftheorem}

There are a number of (conjectural) cases where Kostka polynomials of complex reflection groups satisfy the positivity of their coefficients (\cite{Ma,Sh2,Sh3}). Theorem \ref{freint} supplies a possible framework in which such Kostka polynomials might obtain mathematical reality.

This possibility is supported by the following results that most of the Kostka polynomials in representation theory of reductive groups give rise to Kostka systems by giving graded categorifications of many of their properties:

\begin{ftheorem}[$=$ part of Theorem \ref{gSp} and Corollary \ref{so}]\label{fex}
Every set of Kostka polynomials arising from character sheaves of a connected reductive group admits a realization as a Kostka system whenever the base field is of good characteristic. In addition, such Kostka systems are semi-orthogonal in the sense
\begin{equation}
\mathrm{ext} _{A_W} ^{\bullet} ( K _{\chi}, K _{\eta} ) = \{ 0 \} \hskip 5mm \text{ if } \hskip 5mm \chi < \eta.\label{forh}
\end{equation}
\end{ftheorem}

\begin{fremark}
Note that for a Weyl group of type $\mathsf{A}_n$, the set of Kostka polynomials is unique up to tensoring $\mathsf{sgn}$, while for a Weyl group of type $\mathsf{BC}_n$, we have at least $4(n-1)$ different sets of Kostka polynomials.
\end{fremark}

By a parameter-deformation argument (cf. \cite{L-CG2,Sh5,K1}) and the semi-continuity principle, (\ref{forh}) implies the corresponding $\mathrm{Ext}$-vanishing of the standard modules of a graded Hecke algebra in the sense of \cite{L-CG1} \S 8 (cf. \cite{L-CG2} \S 8 and \cite{CG} \S 8). This also supplies semi-orthogonal collections of many of the Bernstein blocks of $p$-adic groups (cf. \cite{L95,He}).

Since Kostka polynomials in Theorem \ref{fex} are coming from generalized Springer correspondences (\cite{L-IC}), we conclude:

\begin{fcorollary}[$=$ part of Theorem \ref{gSp}]\label{fgSp}
Every twisted total homology group of a generalized Springer fiber $(\cite{L-IC,Lu2})$ is generated by its top-term by hyperplane sections.
\end{fcorollary}

Corollary \ref{fgSp} does not hold for the usual cohomologies in general, and it has been regarded as a mysterious aspect of Springer fibers (cf. \cite{DP,Ta,Ca,GM2,KP}). Hence, our framework provides one reasonable answer to this mystery. Thanks to \cite{BMR,BeM}, Corollary \ref{fgSp} also imposes non-trivial constraints on the structure of modular representation theory of semi-simple Lie algebras and quantum groups.

Kostka systems of the same group are sometimes linked by mutation operations in derived categories. It can also be viewed as a graded analogue of \cite{CKK} \S 3, that is tightly connected with elliptic representation theory ({\it loc. cit.} \S 4). One particular instance is:

\begin{ftheorem}[$\doteq$ part of Theorem \ref{Kmain} + Corollary \ref{omid}]\label{fmain}
Let $\{ K _{\chi} ^{\sharp} \} _{\chi}$ and $\{ K _{\chi} ^{\flat} \} _{\chi}$ be two Kostka systems of type $\mathsf{BC}$ $($arising from character sheaves of connected reductive groups$)$ that have adjacent integral parameter values $($see Lemma $\ref{gSpdata}$ for detail$)$. Then, there exists another Kostka system $\{ K _{\chi} ^{mid} \} _{\chi}$ so that
\begin{itemize}
\item Each of $K _{\chi} ^{mid}$ is written as some extensions of $K _{\chi} ^{\sharp}$ by $K _{\eta} ^{\sharp}$ $(\eta>\chi)$;
\item Each of $K _{\chi} ^{mid}$ is written as some extensions of $K _{\chi} ^{\flat}$ by $K _{\eta} ^{\flat}$ $(\eta<\chi)$.
\end{itemize}
In addition, the Kostka system $\{ K _{\chi} ^{mid} \} _{\chi}$ is also semi-orthogonal.
\end{ftheorem}

Here the expression of Theorem \ref{fmain} is obscure, but we determine exactly which one appears with which grading shift in terms of the notion of strong similarity class (Definition \ref{symBC}) and distance (\S \ref{Wnot}). In addition, we have an explicit description of Kostka systems of type $\mathsf{BC}$ in the asymptotic region ($s \gg 0$ in Example \ref{B2}) in terms of those of type $\mathsf{A}$ (combine Proposition \ref{ak}, Lemma \ref{gind}, and Fact \ref{typeAref} 1)). Therefore, Theorem \ref{fmain} gives an algorithm to compute Kostka polynomials of type $\mathsf{BC}$ (that is independent of the orthogonality relations).

\begin{fexample}\label{B2}
Let $W$ be the Weyl group of type $\mathsf{B}_2$ and consider the total preorders coming from the Lusztig-Slooten symbols with positive parameter range (see \S 4 for detail, but here we warn that our symbols slightly differ from that in \cite{LS}). There are five irreducible representations of $W$
$$\mathsf{sgn}, \mathsf{Ssgn}, \mathsf{Lsgn}, \mathsf{ref}, \mathsf{triv},$$
and the modules $K _{\mathsf{sgn}}$ and $K _{\mathsf{triv}}$ are constant. The transition pattern of the graded characters of the other modules in Kostka systems is:
\begin{center}
\begin{tabular}{c|c|c|c}
$s$ & $\mathsf{gch} \, K _{\mathsf{Lsgn}}$ & $\mathsf{gch} \, K _{\mathsf{Ssgn}}$ & $\mathsf{gch} \, K _{\mathsf{ref}}$ \\ \hline
$s \in ( 0, 1 )$ & $[\mathsf{Lsgn}]$ & $[ \mathsf{Ssgn} ] + t [ \mathsf{ref} ] + t^2 [ \mathsf{triv} ]$ & $[ \mathsf{ref} ] + t [ \mathsf{triv} ] + t [\mathsf{Lsgn}]$ \\
$s = 1$ & $[\mathsf{Lsgn}]$ & $[ \mathsf{Ssgn} ] + t [ \mathsf{ref} ] + t^2 [ \mathsf{triv} ]$ & $[ \mathsf{ref} ] + t [ \mathsf{triv} ]$ \\
$s \in ( 1, 2 )$ & $[ \mathsf{Lsgn} ] + t [ \mathsf{ref} ] + t^2 [ \mathsf{triv} ]$ & $[ \mathsf{Ssgn} ] + t [ \mathsf{ref} ] + t^2 [ \mathsf{triv} ]$ & $[ \mathsf{ref} ] + t [ \mathsf{triv} ]$ \\
$s = 2$ & $[ \mathsf{Lsgn} ] + t [ \mathsf{ref} ] + t^2 [ \mathsf{triv} ]$ & $[ \mathsf{Ssgn} ]$ & $[ \mathsf{ref} ] + t [ \mathsf{triv} ]$ \\
$s > 2$ & $[ \mathsf{Lsgn} ] + t [ \mathsf{ref} ] + t^2 [ \mathsf{triv} ]$ & $[ \mathsf{Ssgn} ]$ & $[ \mathsf{ref} ] + t [ \mathsf{triv} ] + t [ \mathsf{Ssgn} ]$ \\\hline
\end{tabular}
\end{center}
\end{fexample}

The organization of this paper is as follows: The first section is for preliminaries. In \S 2, we define Kostka systems (for complex reflection groups) and present some of their general results. This section is entirely algebraic. In \S 3, we combine the results in \S 2 with Lusztig \cite{L-IC, L-CG2} and Beilinson-Bernstein-Deligne \cite{BBD} to prove that every generalized Springer correspondence gives rise to a Kostka system (Theorem \ref{gSp}). In \S 4, we recall how the description of generalized Springer fibers (of classical types) and symbol combinatorics are related (this part is just a reformulation of known results). In addition, we unify the results of Lusztig \cite{L-CG3} and Opdam-Solleveld \cite{OS} into Slooten's combinatorics (\cite{Sl2}) by utilizing our previous results (\cite{CK, CKK}) and some results from the previous sections. Finally, we present the transition pattern (Theorem \ref{Kmain}) between generalized Springer correspondences of type $\mathsf{BC}$ by utilizing the results from all the previous sections. In the appendices, we provide algebraic proofs that the dual of De Concini-Procesi-Tanisaki \cite{DP, Ta} yields a Kostka system for $W = \mathfrak S _n$, and there exists a Kostka system for $W = \mathfrak S_n \ltimes ( \mathbb Z / 2 \mathbb Z )^n$. Thanks to Garsia-Procesi \cite{GP}, this means that there is a completely algebraic path to study Kostka systems in some cases.

One natural problem arising from this paper is to abstract the arguments so that it include some important non-geometric cases like the Geck-Malle conjecture (\cite{GM}). The author hopes to get back to this problem later.

\begin{flushleft}
{\small
{\bf Acknowledgment:} The author is very grateful to Masaki Kashiwara, Toshiaki Shoji, and Seidai Yasuda for valuable discussions on some technically deep points. The author also thanks Dan Ciubotaru for the collaboration works which leads him to the present paper, and Noriyuki Abe, Pramod Achar, Yoshiyuki Kimura, George Lusztig, Toshio Oshima, and Arun Ram for helpful conversations and correspondences. We have utilized the output of \cite{Ap, GAP} during this research.}
\end{flushleft}

\section{Preliminaries}
\subsection{Overall notation}
Let $(W, S)$ be a complex reflection group with a set of simple reflections and let $\mathfrak h$ be its reflection representation (for $W = \mathfrak S_n$, we might add an additional copy of trivial representation). We form a graded algebra
$$A_W := \mathbb CW \ltimes \mathbb C [ \mathfrak h^* ]$$
by setting $\deg w \equiv 0$ for every $w \in W$ and $\deg \beta = 2$ for every $\beta \in \mathfrak h \subset \mathbb C [ \mathfrak h^* ]$. We set $J_W := \ker \left( \mathbb C [\mathfrak h^*]^W \to \mathbb C \right)$, where the map is the evaluation at $0 \in \mathfrak h^*$. For a subgroup $W' \subset W$, we define $A _{W, W'} := \mathbb C W' \ltimes \mathbb C [ \mathfrak h^* ] \subset A_W$. 

Let $\mathsf{Irr} \, W$ be the set of isomorphism classes of simple $W$-modules, and let $L _{\chi}$ and $e_{\chi}$ be a realization and a minimal idempotent of $W$ corresponding to $\chi \in \mathsf{Irr} \, W$, respectively.

In this paper, every grading should be understood as a $\mathbb Z$-grading. Let $\mathsf{vec}$ be the category of graded vector spaces. Let $A _W \mathchar`-\mathsf{gmod}$ be the category of finitely-generated graded $A_W$-modules. For each $M$ in $A _W \mathchar`-\mathsf{gmod}$ or $\mathsf{vec}$, we denote by $M_i$ its degree $i$ part. We set $M \! \left< d\right>$ to be the grading shift of $M$ of degree $d$ (i.e. $\left( M \! \left< d\right> \right) _i = M _{i-d}$ for each $i \in \mathbb Z$). For $E, F \in A_W \mathchar`-\mathsf{gmod}$ and $R = A_W, \mathbb C [\mathfrak h^*]$, or $W$, we define $\hom_{R} ( E, F )$ to be the direct sum of the space of graded $R$-module homomorphisms $\hom _{R} ( E, F )_j$ of degree $j$. We employ the same notation for extensions (i.e. $\mathrm{ext}_{R} ^i ( E, F ) = \oplus_{j\in \mathbb Z} \mathrm{ext}^{i}_{R} (E,F)_j$). For a graded subspace $J \subset A_W$, we set $\left< J \right>$ to be the (graded) ideal generated by $J$. 

In addition, for $M \in A _W \mathchar`-\mathsf{gmod}$, we define $( M^* ) _{-d} := \Hom _{\mathbb C} ( M _d, \mathbb C)$ and $M ^* := \bigoplus _{d} ( M ^* )_d$. This is a graded $A_W ^{op}$-module that is not necessarily finitely generated. We have an isomorphism $A_W \cong A_W ^{op}$ induced by sending $w \in W$ to $w ^{-1} \in W$ (and is identity on $\mathbb C [ \mathfrak h ^* ]$). Using this, we may also regard $M ^*$ as a (graded) $A_W$-module.

Let $S^d \mathfrak h$ be the $d$-th symmetric power of $\mathfrak h$, which is naturally a $W$-module. In case the reflection representation $\mathfrak h$ of $W$ admits a natural basis $\epsilon_1,\ldots,\epsilon_n$ (as in the case of $W = \mathfrak S_n \ltimes ( \mathbb Z / e \mathbb Z )^n$ for $e \ge 2$), we set $\wedge ^d _+ \mathfrak h \subset S^d \mathfrak h$ to be the span of all the monomials $\epsilon_1 ^{m_1} \epsilon_2^{m_2} \cdots \epsilon_n^{m_n}$ with $0 \le m _i \le 1$ for every $i$. Notice that $\wedge_+^d \mathfrak h \subset S^d \mathfrak h$ is a $W$-submodule.

For $Q (t^{1/2}) \in \mathbb Q (t^{1/2})$, we set $\overline{Q ( t^{1/2} )} := Q ( t^{-1/2} )$.

\subsection{Convention on partitions}\label{Wnot}
Let $\lambda = ( \lambda_1, \lambda_2, \ldots, \lambda _k, \ldots )$ be a non-negative integer sequence such that {\bf 1)} $\sum _i \lambda _i = n$, and {\bf 2)} $\lambda _1 \ge \lambda _2 \ge \cdots \ge 0$. We refer $\lambda$ as a partition of $n$, and $n = |\lambda|$ as the size of $\lambda$. For a partition $\lambda$, we define its transpose partition ${}^{\mathtt t} \lambda$ as $({}^{\mathtt t} \lambda) _i = \# \{ j \mid \lambda _j \ge i \}$. We define $\lambda _k ^{\le} := \sum _{i \le k} \lambda _i$ for each $k \in \mathbb Z _{> 0}$.

We define a partial order on the set of partitions as $\lambda \ge \mu$ if and only if we have $\lambda _k ^{\le} \ge \mu_k ^{\le}$ for every $k$ (for each pair of partitions $\lambda$ and $\mu$). We define the $a$-function of a partition $\lambda$ by $a ( \lambda ) := \sum _{i \ge 1} \left( \begin{matrix} {}^{\mathtt t} ( \lambda ) _i \\ 2 \end{matrix} \right)$. The partial order $<$ is weaker than the partial order given in accordance with the values of the $a$-function (in an opposite way).

For a partition $\lambda$ of $n$, we denote by $\mathfrak S _{\lambda}$ the natural subgroup
$$\mathfrak S _{\lambda _1} \times \mathfrak S _{\lambda _2} \times \cdots \subset \mathfrak S _n.$$
In addition, we have a unique irreducible $\mathfrak S _n$-module $L_{\lambda}$ (up to isomorphism) such that
$$\Hom _{\mathfrak S _{{}^{\mathtt t} \lambda}} (\mathsf{sgn}, L _{\lambda} ) \cong \mathbb C, \text{ and } \Hom _{\mathfrak S _{\lambda}} (\mathsf{triv}, L_{\lambda} ) \cong \mathbb C.$$

A pair of partitions ${\bm \lambda} = (\lambda ^{(0)},\lambda ^{(1)})$ is called a bi-partition, and it is called a bi-partition of $n$ if $n = | \lambda ^{(0)} | + | \lambda ^{(1)} |$ in addition. We denote by $\mathtt P ( n )$ the set of bi-partitions of $n$. The transpose ${} ^{\mathtt t} {\bm \lambda}$ of a bi-partition ${\bm \lambda} = (\lambda ^{(0)},\lambda ^{(1)})$ is defined as $({}^{\mathtt t} \lambda ^{(1)},{}^{\mathtt t} \lambda ^{(0)} )$. We define the $b$-function of a bi-partition ${\bm \lambda}$ as:
$$b ( {\bm \lambda} ) := | \lambda ^{(0)} | + 2 a ( \lambda ^{(0)} ) + 2 a ( \lambda^{(1)}),$$
where we employed the $a$-function of partitions in the RHS.

For a pair of two bi-partitions ${\bm \lambda} = ( \lambda ^{(0)}, \lambda ^{(1)}), {\bm \mu} = ( \mu ^{(0)}, \mu ^{(1)})$ of $n$, we define ${\bm \lambda} \doteq {\bm \mu}$ when there exists a unique pair $(i,j)$ so that $\lambda _i^{(0)} = \mu_i^{(0)} \pm 1$, $\lambda _j^{(1)} = \mu_j ^{(1)} \mp 1$, and $\lambda _k ^{(0)} = \mu _{k} ^{(0)}, \lambda _k ^{(1)} = \mu _{k} ^{(1)}$ otherwise.

For two bi-partitions ${\bm \lambda}$ and ${\bm \mu}$, we define their distance $d_{{\bm \lambda}, {\bm \mu}}$ as:
$$d _{{\bm \lambda}, {\bm \mu}} := \min \{d \mid {\bm \lambda} = {\bm \lambda} _0 \doteq \exists {\bm \lambda} _1 \doteq \cdots \doteq \exists {\bm \lambda} _{d-1} \doteq {\bm \lambda} _d = {\bm \mu} \}.$$

\section{Kostka systems}\label{genkos}
Keep the setting of the previous section.

\begin{lemma}\label{gch-def}
For each $M \in A _W \mathchar`-\mathsf{gmod}$, the following two series belong to $\mathbb Z (\!(t ^{1/2})\!) \mathsf{Irr} \, W$ and $\mathbb Z (\!(t ^{1/2})\!)$, respectively:
$$\mathsf{gch} \, M := \sum _{\chi \in \mathsf{Irr} \, W} \sum _{i \in \mathbb Z} t^{i/2} [L_{\chi}] \dim \Hom _W ( L_{\chi}, M_i ) \text{ and } \mathsf{gdim} \, M := \sum _{i \in \mathbb Z} t^{i/2} \dim M_i.$$
\end{lemma}

\begin{proof}
We have $\dim \, ( A_{W} \left< d \right> )_i = \# W \cdot \dim \, S^{i-d} \mathfrak h < \infty$ for each $i$ and $d$. In addition, we have $\dim \, ( A_{W} \left< d \right> )_i = 0$ if $i<d$. Thus, the assertions hold when $M = A_{W} \left< d \right>$. In general, $M$ is a graded quotient of $\bigoplus _{j \in J} A_W \left< d _j \right>$ (for a finite set $J$ and $d_j \in \mathbb Z$). Therefore, we conclude the assertions by the comparison of their graded pieces.
\end{proof}

Note that $L_{\chi}$ can be regarded as an irreducible $A_W$-module sitting at degree $0$, and we freely use this identification in the below. For each $\chi \in \mathsf{Irr} \, W$, we set $P_{\chi} := A_W e_{\chi}$ and $P _{\chi}^{(0)} := P _{\chi} / \left< J_W \right> P _{\chi}$.

\begin{lemma}\label{cl-indec}
The graded $A_W$-module $P_{\chi}$ is the indecomposable projective cover of $L_{\chi}$. In addition, all finitely generated indecomposable graded projective modules of $A_W$ are of this type up to grading shifts.
\end{lemma}

\begin{proof}
As a direct summand of $A_W$, each $P_{\chi}$ is projective. In addition, we have a natural surjection $P_{\chi} \to L _{\chi}$ with its kernel $\mathfrak h P _{\chi}$. It follows that $P _{\chi}$ is indecomposable, and hence it is a projective cover of $L _{\chi}$. The graded semisimple quotient of $A_W$ is $A_{W,0} = \mathbb C W$. Hence we have an identification of $\mathsf{Irr} \, W$ with the set of isomorphism classes of simple graded $A_W$-modules up to grading shifts. Therefore, $\{ P_{\chi} \}_{\chi}$ exhausts the set of isomorphism classes of indecomposable graded projective modules up to grading shifts.
\end{proof}

\begin{corollary}
The set $\{\mathsf{gch} \, P_{\chi}\} _{\chi \in \mathsf{Irr} \, W}$ is a $\mathbb Z (\!(t^{1/2})\!)$-basis of $\mathbb Z (\!( t^{1/2} )\!) \mathsf{Irr} \, W$.
\end{corollary}

\begin{proof}
For each $\chi \in \mathsf{Irr} \, W$, we have $\mathsf{gch} \, P_{\chi} = [ L _{\chi}] \mod t^{1/2}$. Hence, the linear independence is clear. Every element of $\mathbb Z (\!( t^{1/2} )\!) \mathsf{Irr} \, W$ admits an iterative expansion by $\{ \mathsf{gch} \, P_{\chi} \} _{\chi}$ which removes the lowest (non-zero) graded piece repeatedly. This expansion has finite coefficients at each degree by Lemma \ref{gch-def} as required.
\end{proof}

\begin{proposition}
The category $A_W \mathchar`-\mathsf{gmod}$ has finite projective dimension.
\end{proposition}

\begin{proof}
See McConnell-Robson-Small \cite{MR} 7.5.6.
\end{proof}

Let $K ( A_W )$ be the Grothendieck group of $A_W \mathchar`-\mathsf{gmod}$. We define the graded Euler-Poincar\'e pairing $K (A_W) \times K (A_W) \to \mathbb Z (\!( t^{1/2} )\!)$ as
$$\left< E, F \right> _{\mathsf{gEP}} := \sum _{i \ge 0} (-1) ^i \mathsf{gdim} \, \mathrm{ext} ^i _{A_W} (E,F),$$
where $\mathrm{ext} ^i _{A_W} (E,F) \in \mathsf{vec}$ is the graded extension in $A_W \mathchar`-\mathsf{gmod}$. For each $M \in A_W \mathchar`-\mathsf{gmod}$ and $\chi \in \mathsf{Irr} \, W$, we set
$$[M : L _{\chi}] := \mathsf{gdim} \, \hom _{A_W} (P_{\chi}, M) = \mathsf{gdim} \, \hom _{W} (L_{\chi}, M)$$
and $(M : P_{\chi}) \in \mathbb Z (\!(t^{1/2})\!)$ to be
$$\mathsf{gch} \, M = \sum _{\chi \in \mathsf{Irr} \, W} (M : P_{\chi}) \, \mathsf{gch} \, P_{\chi}.$$

\begin{lemma}
For a finite-dimensional graded $A_W$-module $M$ and $\chi \in \mathsf{Irr} \, W$, we have
$$[M : L _{\chi}] = \overline{[M ^* : L _{\chi ^{\vee}}]}.$$
\end{lemma}

\begin{proof}
By the finite-dimensionality, we have $M ^* \in A_W \mathchar`-\mathsf{gmod}$. The grading of $M ^*$ is opposite to $M$. Therefore, it suffices to prove $( L _{\chi} ) ^* \cong L _{\chi ^{\vee}}$. To this end, it is enough to chase the action of $W$. The naive dual $\mathrm{Hom} _{\mathbb C} ( L_{\chi}, \mathbb C )$ is isomorphic to $L _{\chi ^{\vee}}$ as a $W$-module. This $W$-action factors through $W \subset A_W \cong A_W ^{op}$. Therefore, we conclude the result.
\end{proof}

\begin{definition}[Phyla]
An ordered subdivision
\begin{equation}
\mathsf{Irr} \, W = \mathcal O_1 \sqcup \mathcal O_2 \sqcup \cdots \sqcup \mathcal O _m \label{defphy}
\end{equation}
is called a phyla $\mathcal P = \{ \mathcal O_i \} _{i=1} ^m$ of $W$, and each individual $\mathcal O_i$ is called a phylum. The total preorder $< _{\mathcal P}$ on $\mathsf{Irr} \, W$ defined as
$$\chi < _{\mathcal P} \eta \hskip 2mm \text{ (or } \hskip 1mm \chi \sim _{\mathcal P} \eta \text{)} \hskip 3mm \Leftrightarrow \hskip 3mm \chi \in \mathcal O_{i_1}, \eta \in \mathcal O _{i_2} \hskip 1mm \text{ with } \hskip 1mm i_1 < i_2 \hskip 2mm \text{ (or } \hskip 1mm i_1 = i_2 \text{)}$$
is called the order associated to the phyla $\mathcal P$. If a phyla $\mathcal P$ is fixed, we might drop the subscript $\mathcal P$ from the notation. We define the conjugate phyla $\overline{\mathcal P}$ of $\mathcal P$ by conjugating all irreducible $W$-representations in (\ref{defphy}). We call $\mathcal P$ being of Malle type if $\chi \in \mathcal O_i$ implies $\chi^{\vee} \in \mathcal O_i$, and call $\mathcal P$ a singleton phyla if every phylum is a singleton.
\end{definition}

\begin{remark}\label{rphyla}
{\bf 1)} If $\mathcal P$ is of Malle type, then we have $\overline{\mathcal P} = \mathcal P$. {\bf 2)} If $W$ is a real reflection group, then every phyla is of Malle type since $\chi \cong \chi ^{\vee}$. {\bf 3)} For background about phyla, we refer to Achar \cite{Ac}.
\end{remark}

Let $\Delta := \mathsf{gdim} \, \mathbb C [ \mathfrak h ^* ]^W$. We name $C_{\mathsf{triv}} := P ^{(0)} _{\mathsf{triv}}$.

\begin{lemma}\label{factor}
For each $\chi \in \mathsf{Irr} \, W$, we have $\mathsf{gch} \, P _{\chi} = \Delta \cdot \mathsf{gch} \, P _{\chi} ^{(0)}$. In addition, we have $\dim \, P_{\chi} ^{(0)} < \infty$.
\end{lemma}

\begin{proof}
Since $W$ is a complex reflection group, we have $\dim \, C _{\mathsf{triv}} = \# W < \infty$ by Stanley \cite{St} 4.10. In addition, {\it loc. cit.} 3.1 and 4.1 yields an isomorphism
$$\mathbb C [ \mathfrak h ^* ] \cong C _{\mathsf{triv}} \otimes \mathbb C [ \mathfrak h ^* ]^W$$
as a graded $W$-module. Taking $\mathsf{gch}$ of the both sides and taking account into the fact that $\mathbb C [ \mathfrak h ^* ]^W$ is a direct sum of (infinitely many copies of) ${\mathsf{triv}}$, we conclude
$$\mathsf{gch} \, P _{\mathsf{triv}} = \Delta \cdot \mathsf{gch} \, C _{\mathsf{triv}}.$$
Since $P _{\chi} \cong \mathbb C [\mathfrak h^*] \otimes L_{\chi}$ and $P _{\chi} ^{(0)} \cong C _{\mathsf{triv}} \otimes L_{\chi}$ as graded $W$-modules, we deduce
$$\mathsf{gch} \, P _{\chi} = \Delta \cdot \sum _{\eta \in \mathsf{Irr} \, W} [L_{\eta}] \, \mathsf{gdim} \, \mathrm{hom} _W ( L_{\eta}, C _{\mathsf{triv}} \otimes L _{\chi} ) = \Delta \cdot \mathsf{gch} \, P _{\chi} ^{(0)},$$
which is the first assertion. This also implies $\dim \, P_{\chi} ^{(0)} < \infty$ as required.
\end{proof}

We define the matrix $\Omega$ with its entries
$$\Omega _{\chi,\eta} := \mathsf{gdim} \, \mathrm{hom}_W ( L _{\chi} \otimes L _{\eta ^{\vee}}, C_{\mathsf{triv}} ) \hskip 2mm \text{ for each }\chi, \eta \in \mathsf{Irr} \, W.$$

\begin{corollary}\label{ajEP}
For each $\chi, \eta \in \mathsf{Irr} \, W$, we have $\left< P _{\chi}, P _{\eta} \right> _{\mathsf{gEP}} = \Delta \cdot \Omega_{\chi,\eta}$.
\end{corollary}

\begin{proof}
We have
\begin{align*}
\left< P _{\chi}, P _{\eta} \right> _{\mathsf{gEP}} & = \mathsf{gdim} \, \mathrm{hom} _{A_W} ( P _{\chi}, P _{\eta} )\\
&= \mathsf{gdim} \, \mathrm{hom} _{W} ( L _{\chi}, P _{\eta} ) = \mathsf{gdim} \, \mathrm{hom} _{W} ( L _{\chi}, L _{\eta} \otimes P _{\mathsf{triv}} )\\
&=\Delta \cdot \mathsf{gdim} \, \mathrm{hom} _{W} ( L _{\chi} \otimes L _{\eta ^{\vee}}, C _{\mathsf{triv}} ).
\end{align*}
The last term coincides with $\Delta \cdot \Omega_{\chi,\eta}$ by definition.
\end{proof}

\begin{theorem}[Shoji \cite{Sh1,Sh2}, Lusztig \cite{Lu2}]\label{Sh}
Let $(W, \mathcal P)$ be a pair of a complex reflection group and its phyla. Assume that $K ^{\pm} = ( K ^{\pm} _{\chi,\eta} )_{\chi,\eta \in \mathsf{Irr}\, W}$ are unknown $\mathbb Q (\!(t)\!)$-valued matrices such that
\begin{eqnarray}
K _{\chi,\eta} ^{+} = \begin{cases} 1 & (\chi = \eta)\\ 0 & (\chi \gtrsim \eta \neq \chi) \end{cases}, \hskip 2mm \text{ and } \hskip 2mm \hskip 2mm K _{\chi,\eta} ^{-} = \begin{cases} 1 & (\chi = \eta)\\ 0 & (\chi^{\vee} \gtrsim \eta ^{\vee} \neq \chi^{\vee}) \end{cases}.\label{Kmatrix}
\end{eqnarray}
Let $\Lambda = ( \Lambda _{\chi,\eta} )_{\chi,\eta \in \mathsf{Irr}\, W}$ be also a$(n$ unknown$)$ $\mathbb Q (\!(t)\!)$-valued matrix such that
$$\Lambda _{\chi,\eta} \neq 0 \hskip 2mm \text{ only if } \hskip 2mm \chi \sim \eta.$$
Let $K^{\sigma}$ be the permutation of $K$ by means of $( \chi, \eta ) \mapsto ( \chi^{\vee}, \eta ^{\vee} )$. Then, the matrix equation
\begin{eqnarray}
{} ^{\mathtt{t}} K ^+ \cdot \Lambda \cdot ( K ^- ) ^{\sigma} = \Omega \label{Smatrix}
\end{eqnarray}
has a unique solution.
\end{theorem}

\begin{proof}
We explain how to deduce this from the usual version of the Lusztig-Shoji algorithm (\cite{Sh1, Lu2, Sh2, Ac, AcG}) in the case that $\mathcal P$ is of Malle type (for the sake of simplicity, and in fact otherwise the explanation in the middle does not make sense). We denote $K ^{\pm} _{\chi,\eta}$ by $K _{\chi,\eta}$. Our $K$ is the transpose of the usual convention since our matrix $K$ is designed to represent ``the homology of Springer fibers (cf. \cite{Sp,Sp2,L-IC})" (while usually the matrix $K$ represents the dimensions of the stalks of character sheaves; cf. \cite{BM}). Set $\omega (t) := \mathsf{gch} \, C _{\mathsf{triv}} \in \mathbb Z [t] \mathsf{Irr} \, W$. We have $t^{N^*} \overline{\omega (t)} = \mathsf{gch} \, \left( \mathsf{sgn} \otimes C _{\mathsf{triv}} \right)$, where $N^*$ is the total number of complex reflections of $W$. It implies that our $K _{\chi,\eta}$ are the (unmodified) Kostka polynomials up to normalizations. (Note that \cite{Sh2} \S 5 implies that $K _{\chi,\eta}$ are rational functions for any choice of $\mathcal P$.) Finally, setting $K _{\chi,\chi} = 1$ is achieved by twisting the diagonal matrices (with blockwise same eigenvalues) to $K^{\sigma}, K$, and $\Lambda$, and is a harmless normalization.
\end{proof}

\begin{definition}
For a phyla $\mathcal P$ and $\chi \in \mathsf{Irr} \, W$, we define the $\mathcal P$-trace $P _{\chi, \mathcal P}$ of $P_{\chi}$ (with respect to $\mathcal P$) as
$$P_{\chi, \mathcal P} := P _{\chi} / ( \sum _{\eta \lesssim \chi, f \in \mathrm{hom} _{A_W} ( P_{\eta}, P_{\chi}) _{> 0}} \mathrm{Im} f \, ).$$
\end{definition}

\begin{remark}\label{survive}
{\bf 1)} By the condition $\deg f > 0$, we conclude that $(P_{\chi, \mathcal P}) _0 = L _{\chi}$. {\bf 2)} Since the surjection $P _{\chi} \to P_{\chi, \mathcal P}$ factors through $P^{(0)} _{\chi}$, we deduce that $P_{\chi, \mathcal P}$ is always finite-dimensional. In particular, we have $P_{\chi, \mathcal P} ^* \in A_W \mathchar`-\mathsf{gmod}$.
\end{remark}

\begin{definition}[Kostka systems]\label{kos}
Let $(W, \mathcal P)$ be a pair of a complex reflection group and its phyla. A collection of modules $\mathsf K := \{K_{\chi} ^{\pm}\} _{\chi \in \mathsf{Irr} \, W} \subset A_W \mathchar`-\mathsf{gmod}$ is called a Kostka system (adapted to $\mathcal P$) if it satisfies the following two conditions:
\begin{itemize}
\item[${\bf 1)}$] Each $K_{\chi} ^+$ is a $\mathcal P$-trace of $P _{\chi}$ and each $K_{\chi} ^-$ is a $\overline{\mathcal P}$-trace of $P _{\chi}$;
\item[${\bf 2)}$] We have $\left< K_{\chi} ^+, ( K_{\eta} ^- )^* \right> _{\mathsf{gEP}} \neq 0$ only if $\chi \sim \eta ^{\vee}$.
\end{itemize}
In case $\mathcal P$ is of Malle type, we have $K _{\chi} ^+ = K _{\chi}^-$ for each $\chi \in \mathsf{Irr} \, W$, and we denote them by $K_{\chi}$.
\end{definition}

\begin{problem}\label{Ec}
Does a Kostka system adapted to a (nice) phyla $\mathcal P$ satisfy the orthogonality condition
\begin{itemize}
\item[${\bf 3)}$] $\mathrm{ext} ^{\bullet}_{A_W} ( K_{\chi} ^{\pm}, K_{\eta} ^{\pm} ) \equiv 0$ if $\chi < \eta$ ?
\end{itemize}
Conversely, does a collection of objects in $D^{b} (A_W\mathchar`-\mathsf{gmod})$ with {\bf 3)} and the conditions of Lemma \ref{invert} give rise to a Kostka system whenever their graded characters are positive?
\end{problem}

For more background of Problem \ref{Ec}, see Corollary \ref{so} and Proposition \ref{orthex} in the below.

\begin{lemma}\label{invert}
Let $\{ K ^{+} _{\chi} \} _{\chi}$ and $\{ K ^{-} _{\chi} \} _{\chi}$ be complete collections of $\mathcal P$-traces and $\overline{\mathcal P}$-traces, respectively.
\begin{enumerate}
\item We have $[K_{\chi} ^{\pm} : L_{\eta}] \equiv \delta_{\chi,\eta} \mod t$;
\item We have $[K_{\chi} ^+ : L_{\eta}] \neq 0$ or $[K_{\chi^{\vee}} ^- : L_{\eta ^{\vee}} ] \neq 0$ only if $\chi \lesssim _{\mathcal P} \eta$;
\item We have $[K_{\chi} ^+ : L _{\eta}] \equiv 0 \equiv [K_{\chi^{\vee}} ^- : L_{\eta^{\vee}}]$ if $\chi \sim \eta$ but $\chi \neq \eta$.
\end{enumerate}
\end{lemma}

\begin{proof}
Immediate from the definition of a $\mathcal P$-trace. Notice that we take modulo $t$ in the first assertion instead of $t^{1/2}$ since $[K_{\chi} ^{\pm} : L_{\eta}] \in \mathbb Q [\![ t ]\!]$.
\end{proof}

\begin{proposition}[Problem \ref{Ec} and Kostka systems]\label{orthex}
Let $(W, \mathcal P)$ be a complex reflection group and its phyla. If we have a collection of graded $A_W$-modules $\mathsf K = \{ K _{\chi} ^{\pm} \} _{\chi \in \mathsf{Irr} \, W}$ satisfying the condition of Definition \ref{kos} {\bf 1)} and
\begin{itemize}
\item[${\bf 3)} ^+$] $\mathrm{ext} ^{\bullet} _{A_W} ( K _{\chi} ^+, K_{\eta} ^+ ) = \{ 0 \}$ for every $\chi < _{\mathcal P} \eta$;
\item[${\bf 3)} ^-$] $\mathrm{ext} ^{\bullet} _{A_W} ( K _{\chi} ^-, K_{\eta} ^- ) = \{ 0 \}$ for every $\chi ^{\vee} < _{\mathcal P} \eta ^{\vee}$,
\end{itemize}
then we have
$$\mathrm{ext} _{A_W} ^{\bullet} ( K _{\chi} ^+, ( K_{\eta} ^-)^* ) = \{ 0 \} = \mathrm{ext} _{A_W} ^{\bullet} ( K _{\eta} ^-, ( K_{\chi} ^+ )^* )  \hskip 3mm \text{ unless } \hskip 3mm \chi \sim \eta ^{\vee}.$$
In particular, $\mathsf{K}$ gives rise to a Kostka system.
\end{proposition}

\begin{proof}
By Lemma \ref{invert} and the condition ${\bf 3)}^+$, a repeated use of long exact sequences implies
$$\mathrm{ext} ^{\bullet} _{A_W} ( K _{\chi} ^+, L _{\eta} ) = \{ 0 \} \text{ for every } \chi < _{\mathcal P} \eta.$$
Again by Lemma \ref{invert} and a repeated use of long exact sequences, we deduce
$$\mathrm{ext} ^{\bullet} _{A_W} ( K _{\chi} ^+, ( K _{\eta} ^- ) ^* ) = \{ 0 \} \text{ for every } \chi < _{\mathcal P} \eta ^{\vee}.$$
We have a functorial isomorphism (defined through $A_W \cong A_W ^{op}$)
$$\mathrm{hom} _{A_W} ( M, N ) \cong \mathrm{hom} _{A_W} ( N^*, M^* )$$
for every finite-dimensional graded $A_W$-modules $N,M$. Since $*$ is an exact functor and $\mathrm{ext} _{A_W} ^{\bullet}$ is a universal $\delta$-functor, this implies
$$\mathrm{ext} ^{\bullet} _{A_W} ( K _{\eta} ^-, ( K _{\chi} ^+ ) ^* ) \cong \mathrm{ext} ^{\bullet} _{A_W} ( K _{\chi} ^+, ( K _{\eta} ^- ) ^* ) = \{ 0 \} \text{ for every } \chi < _{\mathcal P} \eta ^{\vee}.$$
By swapping the roles of $K^+$ and $K^-$ by utilizing the condition ${\bf 3)^-}$, we conclude the first assertion. By taking the graded Euler-Poincar\'e characteristic, we deduce the second assertion.
\end{proof}

\begin{theorem}\label{reint}
Assume that we have a Kostka system $\mathsf K$ adapted to $\mathcal P$. Then, the collection $\{ K _{\chi} ^{\pm} \} _{\chi \in \mathsf{Irr} \, W}$ gives rise to the solution of $(\ref{Smatrix})$ as:
$$K _{\chi,\eta} ^{\pm} =  [ K ^{\pm} _{\chi} : L _{\eta} ] \hskip 3mm \text{ for every } \hskip 3mm \chi,\eta\in \mathsf{Irr}\, W.$$
\end{theorem}

\begin{proof}
We define a matrix $P$ with its entries $P_{\chi,\eta} := [ P _{\chi} : L _{\eta} ] \in \mathbb Z [\![t]\!]$. We have $P _{\chi,\eta} \equiv \delta _{\chi,\eta} \mod t$. Therefore, the matrix $P$ is invertible. In addition, we can also regard $P_{\chi,\eta} \in \mathbb Q ( t )$ by Lemma \ref{factor}. By Lemma \ref{invert} and Remark \ref{survive} 2), the same is true for $K^{\pm}$. Hence, we can calculate as:
\begin{align*}
\left< K_{\eta} ^+, ( K_{\chi ^{\vee}} ^- )^* \right> _{\mathsf{gEP}} & = \sum _{\kappa, \nu} \overline{K ^+ _{\eta, \kappa} K_{\chi ^{\vee}, \nu} ^-} \left< L _{\kappa}, L_{\nu^{\vee}} \right> _{\mathsf{gEP}}\\
& = \sum _{\kappa, \nu, \xi} \overline{K _{\eta,\kappa} ^+ K _{\chi ^{\vee}, \nu} ^- (P^{-1}) _{\kappa, \xi}} \left< P _{\xi}, L_{\nu ^{\vee}} \right> _{\mathsf{gEP}}\\
& = \sum _{\kappa, \nu} \overline{K _{\eta, \kappa} ^+ K_{\chi ^{\vee}, \nu} ^- (P^{-1}) _{\kappa, \nu ^{\vee}}}\\
& = ( \overline{K ^+ \cdot P^{-1} \cdot {} ^{\mathtt t} ( K ^- ) ^\sigma} ) _{\eta, \chi}.
\end{align*}
We have $P _{\chi,\eta} = \left< P_{\eta}, P_{\chi} \right> _{\mathsf{gEP}} = \Delta \cdot \Omega_{\eta,\chi}$. Therefore, Definition \ref{kos} {\bf 2)} yields
$${}^{\mathtt t} ( K ^+ \cdot P^{-1} \cdot {} ^{\mathtt t} ( K ^- ) ^{\sigma} ) = \Delta ^{-1} ( ( K ^- ) ^{\sigma} \cdot \Omega ^{-1} \cdot {} ^{\mathtt t} K ^+ ) = \Delta ^{-1} \Lambda ^{-1} \text{ in (\ref{Smatrix}),}$$
as required.
\end{proof}

\begin{corollary}[of the proof of Theorem \ref{reint}]
If we have a collection of $A_W$-modules $\{ K _{\chi} ^{\pm} \} _{\chi \in \mathsf{Irr} \, W}$ so that its graded characters satisfy the equation $(\ref{Smatrix})$ with respect to a phyla, then Definition \ref{kos} {\bf 2)} is satisfied for that phyla. \hfill $\Box$
\end{corollary}

\begin{lemma}[Abe]
For a Kostka system $\mathsf K$ adapted to $\mathcal P$, we have
$$( K _{\chi} ^+ : P_{\eta} ) = 0 \text{ and } ( K _{\chi ^{\vee}} ^- : P_{\eta ^{\vee}} ) = 0 \text{ if } \chi < _{\mathcal P} \eta.$$
\end{lemma}

\begin{proof}
By the linearity of the graded Euler-Poincar\'e pairing, we have
\begin{align*}
\left< K _{\chi} ^+, ( K _{\eta ^{\vee}} ^- ) ^* \right> _{\mathsf{gEP}} & = \sum _{\kappa} \overline{( K _{\chi} ^+ : P_{\kappa}  )} \left< P _{\kappa}, ( K _{\eta ^{\vee}} ^- ) ^* \right> _{\mathsf{gEP}}\\
& = \sum _{\kappa} \overline{( K _{\chi} ^+: P_{\kappa}  )} [ ( K _{\eta ^{\vee}} ^- ) ^* : L _{\kappa} ] \neq 0 \hskip 5mm \text{ only if } \chi \sim \eta.
\end{align*}
Here the matrix $( [ ( K _{\eta ^{\vee}} ^- ) ^* : L _{\kappa} ] )$ is invertible and blockwise upper-triangular (with respect to $\mathcal P$) by Lemma \ref{invert} 1), and the matrix $( \left< K _{\chi} ^+, ( K _{\eta ^{\vee}} ^- ) ^* \right> _{\mathsf{gEP}} )$ is block-diagonal by Definition \ref{kos} {\bf 2)}. Therefore, we conclude the result for $K _{\chi} ^+$. The case of $K _{\chi} ^-$ is similar.
\end{proof}

\begin{proposition}\label{ext}
Let $( W, \mathcal P )$ be a complex reflection group and its phyla. Let $\{ K _{\chi} ^+ \} _{\chi}$ be a complete collection of $\mathcal P$-traces. Then we have
$$\mathrm{ext}^{i} _{A_W} ( K _{\chi} ^+, L _{\eta} ) \cong \mathrm{ext}^{i} _{A_W} ( K_{\eta ^{\vee}} ^-, L _{\chi ^{\vee}} ) \hskip 5mm i = 0,1$$
for every $\chi \sim _{\mathcal P} \eta$, where $K_{\eta ^{\vee}} ^-$ is the $\overline{\mathcal P}$-trace of $P _{\eta ^{\vee}}$.
\end{proposition}

\begin{proof}
Since $\chi \sim _{\mathcal P} \eta$ if and only if $\chi ^{\vee} \sim _{\overline{\mathcal P}} \eta ^{\vee}$, the assertion for $i=0$ is an immediate consequence of the definition of $\mathcal P$-traces.

We prove the case $i = 1$. The first two terms of the minimal projective resolution of $K _{\chi} ^+$ goes as:
$$\bigoplus _{\chi' \in \mathsf{Irr} \, W, d > 0} P _{\chi'} \left< d \right> ^{\bigoplus m_{\chi',d}} \longrightarrow P _{\chi} \longrightarrow K _{\chi} ^+ \to 0.$$
Since $K _{\chi} ^+$ is a $\mathcal P$-trace, we need $\chi' \lesssim \chi$ in order that $m_{\chi',d} \neq 0$. 

Fix an arbitrary $d > 0$. We set $\Gamma _{\chi} ^d := \sum _{f\in \Xi_{\chi} ^d}\mathrm{Im} f \subset P _{\chi}$ and $\Gamma _{\eta ^{\vee}} ^d := \sum _{f\in \Xi_{\eta ^{\vee}} ^d}\mathrm{Im} f \subset P _{\eta ^{\vee}}$, where
$$\Xi_{\chi} ^d = \bigoplus _{\chi' \lesssim \chi, 0 < d' < d} \mathrm{hom}_{A_W} ( P_{\chi'}, P _{\chi} )_{d'},\hskip 1mm \text{ and } \hskip 1mm \Xi_{\eta ^{\vee}} ^d = \bigoplus _{\eta' \lesssim \eta, 0 < d' < d} \mathrm{hom}_{A_W} ( P_{( \eta' )^{\vee}}, P _{\eta ^{\vee}} )_{d'}$$
(here the orderings are taken with respect to the phyla $\mathcal P$). If $m_{\eta,d} \neq 0$, then there exists a $W$-submodule $L _\eta \subset P _{\chi, d}$ that is not contained in $\Gamma _{\chi} ^d$. We identify the dual space $P^* _{\chi}$ with $\mathbb C [ \mathfrak h ] \otimes L _{\chi^{\vee}}$. We have a natural non-degenerate pairing
$$( \bullet, \bullet ) : P _{\chi} \otimes P _{\chi} ^* \longrightarrow \mathbb C$$
induced by a $W$-invariant map $L _{\chi} \otimes L _{\chi ^{\vee}} \to \mathbb C$ and the natural pairing
$$S ^{\bullet} \mathfrak h \times S ^{\bullet} \mathfrak h ^* \ni ( P, f ) \mapsto ( P f ) (0) \in \mathbb C,$$
where we regard $S ^{\bullet} \mathfrak h \cong \mathbb C [ \mathfrak h ^* ]$ as differentials arising from the natural pairing $\mathfrak h^* \times \mathfrak h \to \mathbb C$. In particular, the above pairing equip $P _{\chi} ^*$ a graded $A _W$-module structure, where $\mathfrak h$ acts on $\mathbb C [ \mathfrak h ]$ by derivations.

Let $L$ be the $L_{\eta}$-isotypic part of $P _{\chi,d}$, and let $L^*$ be the $L _{\eta ^{\vee}}$-isotypic part of $P _{\chi,-d}^*$. The natural pairing $( \bullet, \bullet ) : P _{\chi} \times P^* _{\chi} \to \mathbb C$ induces a non-degenerate pairing $( \bullet, \bullet) : L \times L^* \to \mathbb C$. Further, if we write $L \cong L ^{+} \boxtimes L _{\eta}$ and $L ^* \cong L^- \boxtimes L _{\eta ^{\vee}}$ to single out the multiplicity space, then we obtain a non-degenerate pairing $L^+ \times L^- \to \mathbb C$ induced by $L_{\eta} \otimes L_{\eta ^{\vee}} \to \mathbb C$, which we denote by $(\bullet,\bullet)_0$.

For each element $u \in L \cap \Gamma _{\chi} ^d$, we have a non-trivial decomposition
$$u = h_1 u_1 + \cdots + h _m u _m \hskip 5mm \text{ (finite sum)},$$
where $h _i \in \mathbb C [ \mathfrak h^* ]$ is a homogeneous element of degree $( d - d_i )$ and $u _i \in  f_i ( L _{\chi_i} )$ with $f _i \in \hom _{A_W} ( P_{\chi_i}, P _{\chi} )_{d_i} \subset \Xi _{\chi}^d$ for each $1 \le i \le m$. There exists $u' \in L^*$ with $( u, u' ) \neq 0$. It follows that
$$0 \neq \sum _{i = 1} ^m ( h_i u_i, u' ) = \sum _{i = 1} ^m ( u_i, h_i u' ),$$
and hence $(u_{i_0}, h_{i_0} u') \neq 0$ for some $i_0$. Set $d_0 := d_{i_0}$ and $\chi _0 := \chi _{i_0}$. It follows that $\mathbb C W h _{i_0} u'$ contains a $W$-isotypic component $L_{\chi _0 ^{\vee}}$. In particular, we have $u' _0 \in P ^*_{\chi, -d_0}$ so that $( u _{i_0}, u'_{0} ) \neq 0$ and $\mathbb C W u'_0 \cong L _{\chi _0 ^{\vee}}$ by the $W$-invariance of $( \bullet, \bullet )$. We have a decomposition
$$u _{i_0} = h'_1 v_1 + \cdots + h' _{m'} v _{m'} \hskip 5mm \text{ (finite sum)},$$
where $v_i \in L _{\chi} = P _{\chi, 0}$ and $h'_i \in \mathbb C [ \mathfrak h ^* ]$ are degree $d _0$ elements for all $1 \le i \le m'$. By a similar argument as above, there exists $1 \le i_1 \le m'$ so that $( v_{i_1}, h' _{i_1} u'_0 ) \neq 0$.

Let $\sigma _{u'} : P _{\eta ^{\vee}} \left< - d \right> \to P ^* _{\chi}$ be a map determined by $u'$ (i.e. $u' \in \mathrm{Im} \sigma _{u'}$). Let $g_{u'_0} : P _{\chi _0} \left< - d_0 \right> \to P _{\eta ^{\vee}} \left< - d \right>$ be a map obtained by lifting $u'_0$ to $P _{\eta ^{\vee}} \left< - d \right>$ (and require $u'_0 \in \mathrm{Im} g _{u'_0}$). Then the above argument says that for every $u \in L \cap \Gamma _{\chi} ^d$ and every $u' \in L ^*$ with $(u,u') \neq 0$, there exists
$$g_{u'_0} ( h' _{i_1} \otimes u'_{0} ) \in \Gamma _{\eta^{\vee}} ^d \left< - d \right> \subset P _{\eta ^{\vee}} \left< - d \right>$$
so that $\sigma _{u'} ( g_{u'_0} ( h' _{i_1} \otimes u'_{0} ) ) \neq 0$. Notice that the space $L' \boxtimes L_{\chi ^{\vee}}$ of $L_{\chi ^{\vee}}$-isotypic part of $P_{\eta^{\vee}, d}$ is isomorphic to $L^+ \boxtimes L _{\chi ^{\vee}}$ since
\begin{align*}
L^+& \cong \mathrm{hom} _{W} ( L _{\eta}, P _{\chi} )_d \cong \mathrm{hom} _{A_W} ( P _{\eta}, P _{\chi} )_d \cong \mathrm{hom} _{A_W} ( P _{\chi}^*, P _{\eta}^* )_{d}\\
& \cong \mathrm{hom} _{W} ( S^d \mathfrak h^* \otimes L _{\chi ^{\vee}}, L _{\eta ^{\vee}} ) \cong \mathrm{hom} _{W} ( L_{\chi ^{\vee}}, S^d \mathfrak h \otimes L _{\eta ^{\vee}} ) \cong L'.
\end{align*}

Here we have an isomorphism
\begin{align*}
L^- \cong \mathrm{hom} _{W} ( L _{\eta ^{\vee}} \left< - d \right>, P _{\chi} ^* )_0 \cong \mathrm{hom} _{A_W} ( P _{\eta ^{\vee}} \left< - d \right>, P _{\chi} ^* )_0.
\end{align*}

From these, we deduce that for each $u \in L \cap \Gamma _{\chi} ^d$ and $u' \in L^-$ so that $(u,u' \boxtimes L _{\eta} ^{\vee} ) \not\equiv 0$, we have some $u_1 \boxtimes v \in ( L^+ \boxtimes L_{\chi ^{\vee}} \cap \Gamma _{\eta ^{\vee}} ^d )$ so that $(u',u_1)_0 \neq 0$. By taking contraposition, if $u' \in L^-$ satisfies $(u',u_1)_0 = 0$ for every $u_1 \boxtimes v \in ( L^+ \boxtimes L_{\chi ^{\vee}} \cap \Gamma _{\eta^{\vee}} ^d )$, then we have $(u,u' \boxtimes L _{\eta ^{\vee}} ) \equiv 0$ for every $u \in L \cap \Gamma _{\chi} ^d$. 

Therefore, we conclude
$$\hom _W ( ( L \cap \Gamma _{\chi} ^d ), L_{\eta} ) \subset \hom_W ( ( L' \boxtimes L_{\chi ^{\vee}} \cap \Gamma _{\eta ^{\vee}} ^d ), L _{\chi ^{\vee}} ),$$
which is equivalent to a surjective map
$$\mathrm{ext} ^1 _{A_W} ( K _{\chi} ^+, L _{\eta} )_{-d} \longrightarrow \!\!\!\!\! \to \mathrm{ext} ^1 _{A_W} ( K _{\eta^{\vee}} ^-, L _{\chi ^{\vee}} )_{-d}.$$
By the symmetry of the condition, we deduce that this map is actually an isomorphism for each $d > 0$ as desired.
\end{proof}

\begin{corollary}\label{op}
Keep the setting of Proposition \ref{ext}. Let $\mathcal P'$ be another phyla whose total preorder $< _{\mathcal P'}$ is refined by $< _{\mathcal P}$. If we have
$$[ K _{\chi} ^+: L _{\eta} ] = 0 = [ K _{\chi^{\vee}} ^- : L _{\eta ^{\vee}} ] \hskip 5mm \text{ for every } \chi \sim _{\mathcal P'} \eta \text{ but } \chi \not\sim _{\mathcal P} \eta,$$
then $\{ K _{\chi} ^+ \} _{\chi}$ is a complete collection of $\mathcal P'$-traces. In addition, we have
$$\mathrm{ext} ^1 _{A_W} ( K_{\chi} ^+, L _{\eta} ) = \{0\} = \mathrm{ext} ^1 _{A_W} ( K_{\chi^{\vee}} ^-, L _{\eta ^{\vee}} ) \hskip 5mm \text{ for every } \chi \sim _{\mathcal P'} \eta \text{ but } \chi \not\sim _{\mathcal P} \eta.$$
Conversely, let $\mathcal P''$ be a phyla whose total preorder $< _{\mathcal P''}$ refines $< _{\mathcal P}$ and
$$\mathrm{ext} ^1 _{A_W} ( K_{\chi} ^+, L _{\eta} ) = \{0\} = \mathrm{ext} ^1 _{A_W} ( K_{\chi^{\vee}} ^-, L _{\eta ^{\vee}} ) \hskip 5mm \text{ for every } \chi \sim _{\mathcal P} \eta \text{ but } \chi \not\sim _{\mathcal P''} \eta.$$
Then $\{ K _{\chi}^+ \} _{\chi}$ is a complete collection of $\mathcal P''$-traces.
\end{corollary}

\begin{proof}
Observe that the assumption implies
\begin{equation}
[ K_{\chi} ^+ : L_{\eta}] \equiv \delta _{\chi,\eta} \equiv [ K_{\chi ^{\vee}} ^- : L_{\eta ^{\vee}}]  \hskip 5mm \text{ if } \chi \sim _{\mathcal P'} \eta. \label{mult}
\end{equation}
Let $\{ K' _{\chi} \} _{\chi}$ be the (complete) collection of $\mathcal P'$-traces. Each $K' _{\chi}$ is a quotient of $K_{\chi} ^+$ by the images of positive degree map $P_{\chi'} \to K_{\chi} ^+$ for some $\chi \sim _{\mathcal P'} \chi'$, which cannot exist by (\ref{mult}). It follows that $\{ K _{\chi} ^+ \} _{\chi} = \{ K' _{\chi} \} _{\chi}$. The same is true for the collection of $\overline{\mathcal P'}$-traces and $\{ K _{\chi} ^- \} _{\chi}$.

In case $\chi \sim _{\mathcal P'} \eta$ but $\chi \not\sim _{\mathcal P} \eta$, we have either $\chi <_{\mathcal P} \eta$ or $\eta <_{\mathcal P} \chi$. We need to consider only the first case by symmetry. Then, since $K _{\chi} ^+$ is a $\mathcal P$-trace, non-trivial extension of $K _{\chi} ^+$ by $L _{\eta}$ is prohibited. In other words, we have $\mathrm{ext} _{A_W} ^1 ( K_{\chi} ^+, L _{\eta} ) = \{0\}$. Similarly, we have $\mathrm{ext} _{A_W} ^1 ( K_{\chi ^{\vee}} ^-, L _{\eta ^{\vee}} ) = \{0\}$. By Proposition \ref{ext}, we also have $\mathrm{ext} _{A_W} ^1 ( K_{\eta} ^+, L _{\chi} ) = \{0\}$ and $\mathrm{ext} _{A_W} ^1 ( K_{\eta ^{\vee}} ^-, L _{\chi ^{\vee}} ) = \{0\}$. Therefore, we conclude the first assertion. The second assertion is straight-forward.
\end{proof}

\begin{corollary}\label{oc}
Keep the setting of Corollary \ref{op}. If $\{ K _{\chi} ^{\pm} \} _{\chi}$ is a Kostka system adapted to $\mathcal P$, then it is a Kostka system adapted to $\mathcal P'$. In addition, if $\{ K _{\chi} ^{\pm} \} _{\chi}$ is a Kostka system adapted to $\mathcal P$ and
$$\left< K_{\chi} ^+, ( K _{\eta} ^- )^* \right> _{\mathsf{gEP}} = 0 \hskip 5mm \text{ for every } \chi \sim _{\mathcal P} \eta ^{\vee} \text{ but } \chi \not\sim _{\mathcal P''} \eta ^{\vee},$$
then it is a Kostka system adapted to $\mathcal P''$. \hfill $\Box$
\end{corollary}

The following proposition is applied to graded Hecke algebras \cite{Lu3} in a later section.

\begin{proposition}\label{fl}
Let $\mathcal A$ be a $\mathbb C [z]$-algebra with the following properties:
\begin{enumerate}
\item We have an algebra embedding $\mathbb C W \subset \mathcal A$, and $\mathcal A$ is a flat $\mathbb C [z]$-module;
\item Specialization to $z=0$ yields an isomorphism $\mathbb C _0 \otimes _{\mathbb C [z]} \mathcal A \cong A_W$, which identifies subalgebras $\mathbb CW$ in the both sides;
\item There exists a $\mathbb C^{\times}$-action $\mathsf{r} _{\bullet}$ on $\mathcal A$ with $\mathsf{r} _a z = a z$ $(a \in \mathbb C^{\times})$ which induces:
\begin{itemize}
\item an isomorphism $\mathsf{r} _{z_1/z_0}^* : \mathbb C _{z_0} \otimes _{\mathbb C[z]} \mathcal A \stackrel{\cong}{\longrightarrow} \mathbb C _{z_1} \otimes _{\mathbb C[z]} \mathcal A$ for $z_0 \neq 0 \neq z_1$;
\item a dilation action on $A_W = \mathbb C_0 \otimes _{\mathbb C[z]} \mathcal A$ with respect to the grading.
\end{itemize}
\end{enumerate}
Let $M$ be a finite-dimensional irreducible $\mathcal A$-module for which $z$ acts by a nonzero scalar and $L_{\chi}$ appears in $M$ with multiplicity one $($as a $W$-module$)$. Then, there exists an indecomposable graded $A_W$-module $M_0$ $($canonical up to grading shifts and isomorphisms$)$ so that $M \MID _W \cong M_0 \MID _W$ and $P_{\chi}$ surjects onto $M_0$.\\
In addition, if we have a $\mathbb C^{\times}$-equivariant $\mathcal A$-module $\mathcal M$ which is flat over $\mathbb C[z]$ and $M \cong \mathbb C_1 \otimes_{\mathbb C [z]} \mathcal M$, then we have a submodule $\mathcal M ^{\flat} \subset \mathcal M$ so that $\mathbb C[z^{\pm 1}] \otimes _{\mathbb C[z]} \mathcal M ^{\flat} \cong \mathbb C[z^{\pm 1}] \otimes _{\mathbb C[z]} \mathcal M$ and $M_0 \cong \mathbb C_0 \otimes _{\mathbb C[z]} \mathcal M^{\flat}$.
\end{proposition}

\begin{proof}
Suppose that $z$ act by $z_0$ on $M$. By utilizing the $\mathbb C^{\times}$-action, $M$ can be transferred to an $\mathcal A$-module $\mathcal M ^{\circ}$ that is flat over $\mathbb C [z^{\pm 1}]$ and $\mathbb C _{z_1} \otimes _{\mathbb C[z^{\pm 1}]} \mathcal M ^{\circ} \cong \mathsf{r}_{z_1 / z_0}^{*} M$ for each $z_1 \in \mathbb C^{\times}$. Let $\widetilde{P} _{\chi} := \mathcal A e _{\chi}$ be a direct summand of $\mathcal A$. This is a non-zero projective $\mathcal A$-module. By the multiplicity-free assumption and irreducibility, we have a unique (up to scalar multiplications and $z^{\pm 1}$-twists) map $\widetilde{P} _{\chi} \to \mathcal M ^{\circ}$ which becomes surjection after localizing to $\mathbb C [z^{\pm 1}]$. Let $\mathcal K$ be the kernel of this map, which is an $\mathcal A$-submodule of $\widetilde{P} _{\chi}$ by definition. Here $\mathcal K$ must be a torsion-free $\mathbb C [ z ]$-module since $\widetilde{P} _{\chi}$ is so. Here $\mathbb C[z]$ is PID, so $\mathcal K$ is flat as a $\mathbb C [z]$-module. Therefore, we have inclusions of $\mathcal A$-modules
$$\mathcal K \subset \mathcal K' := \mathbb C[z^{\pm 1}] \otimes _{\mathbb C [z]} \mathcal K \cap \widetilde{P} _{\chi} \subset \mathbb C [z^{\pm 1}] \otimes _{\mathbb C [z]} \widetilde{P} _{\chi}.$$
By the maximality of this module and again by fact that $\mathbb C [z]$ is PID, we conclude that $\widetilde{P} _{\chi} / \mathcal K'$ is flat as a $\mathbb C [ z ]$-module. In addition, $\mathcal M ^{\circ}$ and $\widetilde{P} _{\chi} / \mathcal K'$ are naturally isomorphic if we invert $z$. By the rigidity of (finite-dimensional) $W$-modules, we conclude that $M_0 := \mathbb C _0 \otimes _{\mathbb C [z]} ( \widetilde{P} _{\chi} / \mathcal K' )$ has the same $W$-module structure as that of $M$. In addition, it admits a surjection from $P_{\chi} \cong \mathbb C_0 \otimes _{\mathbb C[z]} \widetilde{P} _{\chi}$. Now we utilize the $\mathbb C^{\times}$-action to deduce $M_0$ is graded.

For the latter assertion, we set $\mathcal M^{\flat} := ( \widetilde{P} _{\chi} / \mathcal K' )$. We rearrange the above map by twisting some power of $z$ if necessary to obtain a homomorphism $\widetilde{P} _{\chi} \longrightarrow \mathcal M$, whose image contains $\mathbb C[z] W e _{\chi} \cong \mathbb C[z] L_{\chi}$. By the above construction, it gives rise to a submodule $\mathcal M^{\flat} \subset \mathcal M$ as desired.
\end{proof}

\section{Kostka systems arising from reductive groups}
We use the setting of the previous section. In this section, we prove the existence of a Kostka system corresponding to a generalized Springer correspondence by utilizing Lusztig's construction of generalized Springer correspondence/graded Hecke algebra.

In this section (and only in this section), we work over a field of positive characteristic in order to apply the machinery of \cite{BBD}. We fix two distinct primes $p$ and $\ell$, set $\mathbb F$ to be a finite extension of $\mathbb F_p$, and set $\Bbbk$ to be the algebraic closure of $\mathbb F$. We define $\mathsf{Fr}$ to be the geometric Frobenius morphism such that $X ( \Bbbk ) ^{\mathsf{Fr}} = X ( \mathbb F )$ for a variety $X$ over $\mathbb F$. For sheaves, we usually work in the derived category, and hence we understand that all functors are derived unless stated otherwise. We utilize some identification $\overline{\mathbb Q} _{\ell} \cong \mathbb C$ to pass the results to the other cases.

A generalized Springer correspondence is determined by the following data (\cite{L-IC}): a split connected reductive group $G$ over $\mathbb F$, its split Levi subgroup $L$, a cuspidal $\overline{\mathbb Q}_{\ell}$-local system $\mathcal L$ on a nilpotent orbit $\mathcal O_c$ of $L$, and its Frobenius linearization $\phi : \mathsf{Fr} ^* \mathcal L \stackrel{\cong}{\longrightarrow} \mathcal L$ (which is a descent data from $\Bbbk$ to $\mathbb F$) defined over $\mathcal O _c \otimes _{\mathbb F} \Bbbk$. We call $\mathbf c := ( G, L, \mathcal O_c, \mathcal L, \phi )$ a {\it cuspidal datum}.

We assume that the characteristic of $\mathbb F$ is good for $G$. For an algebraic group, we denote its Lie algebra by its small gothic letter. Let $\mathcal N _{G} \subset \mathfrak g$ denote the nilpotent cone of $G$. Let $P \subset G$ be a parabolic subgroup of $G$, with a choice of its Levi decomposition $P = LU$. The nilpotent cone $\mathcal N_L = \mathcal N_G \cap \mathfrak l$ of $L$ contains the $L$-orbit $\mathcal O_c$. Form a collapsing map
$$\mu : G \times ^{P} ( \overline{\mathcal O_c} \oplus \mathfrak u ) \longrightarrow \mathcal N _G.$$
We denote the domain of $\mu$ by $\widetilde{\mathcal N}$, and the image of $\mu$ by $\mathcal N$. Note that $\mu$ is proper and $\mathcal N$ is closed in $\mathcal N _G$. Let $j : \mathcal O_c \to \overline{\mathcal O_c}$ be the natural inclusion map and let $\mathsf{pr} : ( \overline{\mathcal O_c} \oplus \mathfrak u ) \to \overline{\mathcal O_c}$ be the projection map. They are $L$- and $P$-equivariant, respectively. By the cleanness property of cuspidal local systems (Ostrik \cite{Os}), we have $j _! \mathcal L \cong j _* \mathcal L$, and hence $\mathsf{pr} ^* j _! \mathcal L$ defines a (shifted) $P$-equivariant perverse sheaf on $( \overline{\mathcal O_c} \oplus \mathfrak u )$. By taking the $G$-translation, we obtain a (shifted) $G$-equivariant perverse sheaf $\dot{\mathcal L}$ on $G \times ^{P} ( \overline{\mathcal O_c} \oplus \mathfrak u )$. Let $W = W _{\mathbf c}:= \{ g \in N_G(L) \mid g^* \mathcal L \cong \mathcal L \} / L$ be the Weyl group attached to $\mathbf c$. Let $H^{\circ}$ be the identity component of an algebraic group $H$. For $x \in \mathcal N ( \mathbb F )$, we set $A_x := Z_G(x)/Z_G(x)^{\circ}$.

The following Theorem \ref{LusSTD} is (logically) buried in Lusztig \cite{L-IC, Lu2, L-CG1, L-CG2} (which lies on the results of many mathematicians, including those of Borho-MacPherson \cite{BM}, Ginzburg \cite{Gi,CG}, Shoji \cite{Sh1}, Beynon-Spaltenstein \cite{BS}, and Evens-Mirkovi\'c \cite{EM}). Some part of its Lie algebra version is presented in Letellier \cite{Let} \S 5 (which serves a good point to begin with) and Mirkovi\'c \cite{Mi}. Hence, all the assertions in Theorem \ref{LusSTD} are known to experts, and the author is claiming {\it no} originality for Theorem \ref{LusSTD} itself. Nevertheless, we provide explanations on how to deduce the present form for the sake of completeness.

\begin{theorem}[Lusztig's generalized Springer correspondence]\label{LusSTD} We have the following results over $\Bbbk$:
\begin{enumerate}
\item The sheaf $\mu _* \dot{\mathcal L} [\dim \widetilde{\mathcal N}]$ is perverse, and is a direct sum of simple perverse sheaves $($with respect to the self-dual perversity$)$;
\item We have $A_{W} \cong \mathrm{Ext} ^{\bullet} _G ( \mu _* \dot{\mathcal L},  \mu _* \dot{\mathcal L} )$ as graded algebras, where the extension is taken in the $G$-equivariant derived category $D_G^b (\mathcal N)$;
\item {\bf (generalized Springer correspondence)} For each $\chi \in \mathsf{Irr} \, W$, there exists a simple $(G$-equivariant$)$ perverse sheaf $\mathsf{IC} (\chi)$ on $\mathcal N$ so that:
\begin{equation}
\mu _* \dot{\mathcal L} [\dim \widetilde{\mathcal N}] \cong \bigoplus _{\chi \in \mathsf{Irr} \, W} L_{\chi} \boxtimes \mathsf{IC} ( \chi ).\label{IC-simple}
\end{equation}
In addition, we have $\mathsf{IC} ( \chi ) \cong \mathsf{IC} ( \chi' )$ if and only if $L _{\chi} \cong L _{\chi'}$ as $W$-modules;
\item For each $i \in \mathbb Z$, the Frobenius action $($arising from $\phi)$ of $\mathrm{Ext} ^{i} _G ( \mu _* \dot{\mathcal L},  \mu _* \dot{\mathcal L} )$ is pure of weight $i$. More precisely, $\phi$ induces a vector space automorphism with the absolute values of all of its eigenvalues equal to $q ^{i/2}$;
\item For each $x \in \mathcal N ( \mathbb F )$, we set $\mathfrak B_x := \mu ^{-1} (x)$ and $\imath _x : \{ x \} \hookrightarrow \mathcal N$. Then, the graded vector space
\begin{equation}
H_{\bullet} ( \mathfrak B_x, \dot{\mathcal L} ) := \mathbb H ^{\bullet} ( \imath_x ^! \mu _* \dot{\mathcal L} [2 \dim \mathcal N - 2 \dim \mathfrak B_x ] ) \label{homology}
\end{equation}
admits a structure of a graded $A_W$-module which commutes with the $A_x$-action;
\item Let $x \in \mathcal N (\mathbb F)$. For each $\xi \in \mathsf{Irr} \, A_x$, we define
$$K^{{\mathbf c}, gen} _{(x, \xi)} = H_{\bullet} ( \mathfrak B_x, \dot{\mathcal L} ) _{\xi} := \Hom _{A_x} (\xi, H_{\bullet} ( \mathfrak B_x, \dot{\mathcal L} ) )$$
and call it the generalized Springer representation. The graded module $K^{{\mathbf c}, gen} _{(x, \xi)}$ is concentrated in non-negative even degrees;
\item Fix $x \in \mathcal N (\mathbb F)$ and let $\xi \in \mathsf{Irr} \, A_x$. For every $\chi' \in \mathsf{Irr} \, W$, we have
$$[K _{(x,\xi)} ^{\mathbf c, gen} : L _{\chi'}] = t ^{\dim \mathfrak B _x - \frac{1}{2} \dim \mathcal N} \mathsf{gdim} \, \mathrm{Hom} _{A_x} ( \xi, \mathbb H ^{\bullet} ( i _{x} ^! \mathsf{IC} (\chi') ) );$$
\item Each $x \in \mathcal N ( \mathbb F )$ and $\xi \in \mathsf{Irr} \, A_x$ gives rise to a $G$-equivariant simple perverse sheaf $\mathsf{IC} ( x, \xi )$ via the minimal extension of the local system on $G.x$ corresponding to $\xi$. If this $\mathsf{IC} ( x, \xi )$ is not of the form $\mathsf{IC} ( \chi )$ for some $\chi \in \mathsf{Irr} \, W$, then $K^{{\mathbf c}, gen} _{(x, \xi)} = \{ 0 \}$;
\item If $K^{{\mathbf c}, gen} _{(x, \xi)} \neq \{ 0 \}$, then $( K^{{\mathbf c}, gen} _{(x, \xi)} ) _0$ is irreducible as a $W$-module. In addition, the Frobenius action on $K^{{\mathbf c}, gen} _{(x, \xi)}$ is pure;
\item The graded $W$-module $K^{{\mathbf c}, gen} _{(x, \xi)}$ is isomorphic to the one defined by using varieties over $\mathbb C$.
\end{enumerate}
\end{theorem}

\begin{remark}\label{BLex}
{\bf 1)} For the sake of simplicity, our homologies substantially differ from the usual convention (e.g. their degrees are cohomological). In particular, the $i$-th homology of a smooth irreducible variety $\mathfrak X$ (in this paper) is $H ^{i - 2 \dim \mathfrak X} ( \mathfrak X, \mathbb D _{\mathfrak X} )$, where $\mathbb D_{\mathfrak X}$ is the dualizing sheaf of $\mathfrak X$. {\bf 2)} There are other Springer correspondences (see e.g. Xue \cite{Xu}). It might be interesting to see whether they give rise to Kostka systems, and how they are related with those in this paper.
\end{remark}

\begin{proof}[Sketch of the proof of Theorem \ref{LusSTD}]
Here we use the good characteristic assumption in several ways: One is to utilize the Springer isomorphism between the unipotent variety and the nilpotent cone of $G$. Another is to assume the set of nilpotent orbits, its dimensions, its stabilizers at points, and its closure relations are in common between over $\mathbb F$ and over $\mathbb C$ (\cite{CM}). The other is that \cite{L-IC,Lu2} sometimes requires the good characteristic assumption.

{\bf 1)} follows from \cite{L-IC} 6.5c. Since $\mathbf H$ in \cite{L-CG2} 8.11 is free over $H ^{\bullet} _{\mathbb G_m} (\mathrm{pt})$, the forgetful map must be surjective by the Serre spectral sequence. We have $A_W \cong \mathbf H / ( \mathbf r )$ in the notation of \cite{L-CG2} \S 8. Therefore, {\bf 2)} follows from the positive characteristic analogue of \cite{L-CG2} 8.11. For its proof (\cite{L-CG2}, or the combination of \cite{L-CG1} and \cite{CG} \S 8.6) to work in our setting (and to justify the proof of {\bf 4)}), it suffices to have a model of $EG$ defined over $\mathbb F$ which yields the mixed version $D ^b _{G,m} ( \mathcal N )$ of $D ^b _{G} ( \mathcal N )$.

In \cite{L-CG1,L-CG2}, the space $EG$ is replaced by a smooth irreducible variety $\Gamma$ (depending on $j$) with a free $G$-action and $H^m ( \Gamma ) = \{ 0 \}$ for $0 < m \le j$ (to compute the $j$-th $G$-equivariant cohomology). The weight structure of $H^j ( BG ) = H ^j ( G \backslash \Gamma )$ is independent of the choice of such $\Gamma$.

Hence, the Borel approximation model of $EG$ (cf. \cite{L-CG1} 1.1) yield the (well-defined) notion of weights in $G$-equivariant cohomologies. This implies the existence of $D ^b _{G,m} ( \mathcal N )$. Thus, {\bf 2)} follows by \cite{L-CG2}, or by \cite{L-CG1} and \cite{CG}. See Shoji \cite{Sh5} \S 2 for more detailed justification (which covers \cite{L-CG1}).

The sheaf $\dot{\mathcal L}$ is of geometric origin (\cite{BBD} 6.2.4) by the classification of cuspidal pairs in \cite{L-IC}. Since $\mu$ is proper, it follows that $\mu _* \dot{\mathcal L}$ is a direct sum of simple perverse sheaves (\cite{BBD} 5.4.6). The presentation of $A_{W,0}$ implies $\overline{\mathbb Q}_{\ell} W \cong \mathrm{Hom} _G ( \mu _* \dot{\mathcal L}, \mu _* \dot{\mathcal L} )$. Therefore, the rest of the assertions in {\bf 3)} follows.

The vector space $\overline{\mathbb Q}_{\ell} W = A_{W,0}$ is pure of weight $0$ (since $W$ arise as automorphisms of $\mu _* \dot{\mathcal L}$ in $D ^b_G ( \mathcal N )$, and is defined over $\mathbb F$), and $H^2 _L ( \mathcal O_c )$ is pure of weight $2$ (actually $\phi$ induces $q \mathrm{id}$, since our groups $G, L$, and $Z ( L ) ^{\circ}$ are $\mathbb F$-split by assumption). Since $A_W$ is generated by $\overline{\mathbb Q}_{\ell} W$ and $H^{2} _L ( \mathcal O_c )$ (by \cite{L-CG1} 4.1, 5.1 and \cite{L-CG2} 8.11), we deduce {\bf 4)}.

With {\bf 1)} and {\bf 2)} in hands, {\bf 5)} follows by \cite{L-CG1} 8.1, 8.2 (base change is applicable by the cleanness property of $\mathcal L$). The non-negativity assertion of {\bf 6)} follows by the vanishing costalk condition in the definition of perverse sheaves applied to {\bf 1)} and the fact that $\mu$ is semi-small by \cite{L-IC} 1.2. The evenness assertion of {\bf 6)} follows by \cite{Lu2} 24.8a and the fact that every nilpotent orbit has even dimension (\cite{CM}).

The $W$-module structure of $K^{{\mathbf c}, gen} _{(x, \xi)}$ arises from $A_{W,0}$ (cf. \cite{L-CG1} 8.1). Therefore, (\ref{IC-simple}) implies that the $L_{\chi'}$-isotypic part of $H_{\bullet} ( \mathfrak B_x, \dot{\mathcal L} )$ given by $\mathbb H^{\bullet} ( i_x^! ( L_{\chi'} \boxtimes \mathsf{IC} ( \chi' ) ) [ \dim \mathcal N - 2 \dim \mathfrak B_x ])$. This yields {\bf 7)}.

In view of {\bf 1)} and {\bf 5)}, \cite{Lu2} 24.8c implies {\bf 8)}. The first part of {\bf 9)} follows by (\ref{IC-simple}) and the vanishing costalk condition of simple perverse sheaves. The latter half of {\bf 9)} follows by \cite{Lu2} 24.6.

We explain {\bf 10)}. By \cite{Lu2} 24.8b, we deduce that the dimensions of the stalks of $G$-equivariant perverse sheaves are in common between all good characteristics. We utilize \cite{BBD} (6.1.10.1) to conclude that they are also in common with that over $\mathbb C$. In addition, $\mu_* \dot{\mathcal L}$ is of geometric origin. In particular, simple perverse sheaves appearing in $\mu_* \dot{\mathcal L}$ are in common between over $\mathbb F$ (provided if the characteristic is large enough) and over $\mathbb C$ (\cite{BBD} 6.2.2--6.2.7). These are enough to deduce the assertion from the definition (\ref{homology}).
\end{proof}

We denote the degree zero part of $K^{{\mathbf c}, gen} _{(x, \xi)}$ (if non-zero) by $L_{(x,\xi)}$. If $L_{(x,\xi)} \cong L_{\chi}$ as a $W$-module, then we call $(x,\xi)$ the Springer correspondent of $\chi$ with respect to $\mathbf c$. This is equivalent to $\mathsf{IC} (x,\xi) \cong \mathsf{IC} ( \chi )$. For each $\chi \in \mathsf{Irr} \, W$ with its Springer correspondent $(x, \xi)$, we set $\mathcal O _{\chi} := G. x \subset \mathcal N$.

\begin{theorem}\label{injectivity}
Fix a phyla $\mathcal P$ that is a refinement of the closure ordering of the generalized Springer correspondence attached to $\mathbf c$. Then, $K^{{\mathbf c}, gen} _{(x, \xi)}$ is the $\mathcal P$-trace of $L _{(x,\xi)}$.
\end{theorem}

\begin{proof}
Fix $\chi \in \mathsf{Irr}\, W$ so that $(x,\xi)$ is the Springer correspondent of $\chi$. We denote $\mathcal O _{\chi}$ by $\mathcal O$ for the sake of simplicity. Let $\imath : \mathcal O \hookrightarrow \mathcal N$ be the inclusion. We set $d := \dim \widetilde{\mathcal N} = \dim \mathcal N$. We set $\ddot{\mathcal L} := \mu _* \dot{\mathcal L} [d] ( \frac{d}{2})$. Here $(\frac{d}{2})$ is the Tate twist which makes $\ddot{\mathcal L}$ perverse and pure of weight $0$ (cf. \cite{BBD} 5.1.8, 5.4.5, and 5.4.9. Note that here we interpret that the Tate twist has an effect on the data $\phi$ which we omitted from the notation).

By Theorem \ref{LusSTD} 2) and (\ref{IC-simple}), we have
$$P_{\chi} = A_W e _{\chi} \cong \mathrm{Ext} ^{\bullet} _G  ( \mathsf{IC} ( \chi ), \ddot{\mathcal L} ).$$
We set $\mathcal E := \imath ^* \mathsf{IC} ( \chi )$ and write $\mathcal E^! := \imath _! \mathcal E$. Let $\imath _y : \{ y \} \hookrightarrow \mathcal N$ be the inclusion of $y \in \mathcal N ( \mathbb F )$. The $A_W$-module $\mathrm{Ext} ^{\bullet} _G  ( \mathcal E^!, \ddot{\mathcal L} )$ is rewritten as:
\begin{align}\nonumber
& \mathrm{Ext} ^{\bullet} _G  ( \mathcal E^!, \ddot{\mathcal L} ) \cong \mathrm{Ext} ^{\bullet} _G  ( \mathcal E , \imath ^! \ddot{\mathcal L} ) \cong \mathrm{Ext} ^{\bullet} _{Z_G(x)}  ( \xi, \imath _x ^! \ddot{\mathcal L} ) \hskip 4mm \text{(adjunction and \cite{BL} 2.6.2)} \\\nonumber
& \cong \mathrm{Ext} ^{\bullet}_{A_x} ( \xi, \mathrm{Ext} ^{\bullet} _{Z_G(x) ^{\circ}} ( \overline{\mathbb Q} _{\ell}, \imath _x ^! \ddot{\mathcal L} ) ) \cong H_{\bullet} ^{Z_G(x) ^{\circ}} ( \mathfrak B_x, \dot{\mathcal L} ) _{\xi}\\
& = \bigoplus _{\zeta \in \mathsf{Irr} A_x} \Hom _{A_x} ( \xi, H ^{\bullet} _{Z_G(x) ^{\circ}} ( \{ x \} ) \otimes H_{\bullet} ( \mathfrak B_x, \dot{\mathcal L} ) _{\zeta} ),\label{gSpfiber}
\end{align}
where we utilized the fact that $\mathrm{Ext} ^{\bullet}_{A_x} ( \overline{\mathbb Q} _{\ell}, \bullet ) = \mathrm{Ext} ^{\bullet}_{D^b_{A_x} ( \mathrm{Spec} \, \Bbbk )} ( \overline{\mathbb Q} _{\ell}, \bullet )$ is the functor taking the $A_x$-fixed part of (a complex of) vector spaces. We set
$$\Lambda := \{ \eta \in \mathsf{Irr} \, W \MID \mathcal O _{\eta} \subset \overline{\mathcal O} \backslash \mathcal O \}.$$

We denote by ${}^p H ^{\bullet}$ and $\tau _{\bullet}$ the perverse cohomology functor and the truncation functor of $D ^b _G ( \mathcal N )$ with respect to its (self-dual) perverse $t$-structure. Then, the right $t$-exactness of $\imath _!$ implies
$${}^p H ^i ( \mathcal E ^! ) \neq \{ 0 \} \hskip 2mm \text{ only if } i \le 0.$$
Thanks to Theorem \ref{LusSTD} 2), we deduce an isomorphism
$$\mathrm{Ext} ^{odd}_G ( \mathsf{IC} ( \chi' ), \mathsf{IC} ( \chi'' ) ) = \{ 0 \} \text{ for every } \chi',\chi'' \in \mathsf{Irr} \, W.$$

In order to apply the formalism of weights, we sometimes descend from $\Bbbk$ to $\mathbb F$ by means of a Frobenius linearization. In particular, we understand that if a sheaf $\mathcal F$ is defined over $\Bbbk$, then $\mathcal F _{0}$ is the corresponding sheaf defined over $\mathbb F$ by utilizing the Frobenius linearization (coming from $\phi$ in $\mathbf c$). Thanks to the edge exact sequence
\begin{align}
0 \to \mathrm{Hom} _G ( \mathsf{IC} ( \chi' ), \mathsf{IC} ( \chi'' ) )_{\mathsf{Fr}} \to \mathrm{Ext} ^1 _G ( \mathsf{IC} ( \chi' )_0, \mathsf{IC} ( \chi'' )_0 ) \to \mathrm{Ext} ^1 _G ( \mathsf{IC} ( \chi' ), \mathsf{IC} ( \chi'' ) ) ^{\mathsf{Fr}} \to 0,\label{Fext}
\end{align}
we conclude that each ${}^p H ^i ( \mathcal E ^! ) _0$ is a direct sum of simple $G$-equivariant perverse sheaves (up to extensions between Tate twists of isomorphic modules) provided if all the constituents are of the form $\mathsf{IC} ( \chi' ) _0$ for some $\chi' \in \mathsf{Irr} \, W$.

We have a surjection
$${}^p H^0 ( \mathcal E ^! ) _0 \longrightarrow \!\!\!\!\! \rightarrow \mathsf{IC} ( \chi ) _0$$
in the category of perverse sheaves, which is a unique simple quotient.

\begin{claim}\label{j-ext}
We have ${}^p H ^0 ( \mathcal E ^! ) _0 = \mathsf{IC} ( \chi ) _0.$
\end{claim}

\begin{claim}\label{comp-anal}
For each $i < 0$, a direct summand of ${}^p H ^i ( \mathcal E ^! )_0$ is of the form $V_{\eta} \boxtimes \mathsf{IC} ( \eta )_0$ for some $\eta \in \Lambda$ and some continuous $\mathrm{Gal} ( \Bbbk / \mathbb F)$-module $V _{\eta}$. In addition, it is mixed of weight $< i$.
\end{claim}

\begin{proof}[Proof of Claims \ref{j-ext} and \ref{comp-anal}]
We prove the assertions by induction. For each $k \ge 0$, we denote by $j _k : \mathbb O _{k} \hookrightarrow \mathcal N$ the embedding of the union of all $G$-orbits of dimension $\ge \dim \mathcal O - k$. We set $\mathbb O' _k := \mathbb O_k \backslash \mathbb O_{k-1}$. We define $\jmath _k : \mathbb O _{k-1} \hookrightarrow \mathbb O_{k}$. It is clear that $j _k$ and $\jmath _k$ are open embeddings for each $k \ge 0$. We prove the assertions by induction on $k$.

We suppose that the assertions are true when restricted to $\mathbb O_{k-1}$. Notice that $\mathcal O \subset \mathbb O_0$ is a closed subset and hence the assertion holds when restricted to $\mathbb O_0$. We need to show that the assertions hold when restricted to $\mathbb O_k$.

By induction hypothesis, we have
$${}^p H ^i ( j _{k-1} ^! \mathcal E^! )_0 = \begin{cases} \{0\} & (i>0)\\ j _{k-1} ^! \mathsf{IC} ( \chi ) _0 & (i=0)\end{cases}$$
and each direct summand of ${}^p H ^i ( j _{k-1} ^! \mathcal E^! )_0$ ($i < 0$) is of the form $V_{\eta} \boxtimes j _{k-1} ^! \mathsf{IC} ( \eta )_0 = V_{\eta} \boxtimes j _{k-1} ^* \mathsf{IC} ( \chi )_0$ for some $\eta \in \Lambda$ with its weight $< i$.

We consider the distinguished triangle
$$\to ( \mathcal K_i )_0 \to ( \jmath _k )_{!} {}^p H ^{i} ( j _{k-1} ^! \mathcal E^! )_0 [-i] \to ( \jmath _k )_{!*} {}^p H ^{i} ( j _{k-1} ^! \mathcal E^! )_0 [-i] \stackrel{+1}{\longrightarrow},$$
where $( \jmath _k )_{!*}$ denote the minimal extension. The stalk of $( \jmath _k )_{!} {}^p H ^{i} ( j _{k-1} ^! \mathcal E^! )_0$ is zero along $\mathbb O_{k}'$ (by definition). For each $y \in \mathbb O_k' ( \mathbb F )$, we deduce that
\begin{equation}
\imath _y ^* H ^m ( ( \jmath _k )_{!*} {}^p H ^{i} ( j _{k-1} ^! \mathcal E^! )_0 [-i] ) \cong \imath _y ^* H ^{m+1} ( ( \mathcal K_i )_0 ) \hskip 5mm \text{ for each } m.\label{dist-succ}
\end{equation}
This implies that the pointwise weight of $( \mathcal K_i )_0$ is exactly one less than that of $( \jmath _k )_{!*} {}^p H ^{i} ( j _{k-1} ^! \mathcal E^! )_0 [-i]$ along $\mathbb O_{k}' ( \mathbb F )$. Therefore, all simple perverse sheaves supported on $\mathbb O'_k$ appearing in ${}^p H^m ( ( \jmath _k )_{!} {}^p H ^{i} ( j _{k-1} ^! \mathcal E^! ) ) _0$ must have weight $< ( m + i - 1 )$ ($m + i < 0$) or weight $< 0$ ($i = 0 = m$). Utilizing \cite{BBD} 5.4.1 (and the argument just after that), we deduce that ${}^p H^i ( j_k ^! \mathcal E^! ) _0$ has weight $< i$ for each $i < 0$. Now each ${}^p H ^{m} ( \mathcal K_i ) _0$ acquires only the sheaves of the form $j_k ^! \mathsf{IC} (\eta) _0$ for $\eta \in \Lambda$ (up to Tate twists) by the comparison of the stalks by using Theorem \ref{LusSTD} 8) and the induction hypothesis. This implies ${}^p H^0 ( j _{k} ^! \mathcal E^! ) _0 \cong {}^p H^0 ( ( \jmath _k )_{!} {}^p H ^{0} ( j _{k-1} ^! \mathcal E^! )) _0 \cong j _{k} ^! \mathsf{IC} ( \chi ) _0$ and every Jordan-H\"older constituent of ${}^p H^i ( j_k ^! \mathcal E^! )$ ($i < 0$) is of the form $j_k ^! \mathsf{IC} ( \eta )$ for some $\eta \in \Lambda$. Therefore, the induction proceeds and we conclude the results.
\end{proof}

We return to the proof of Theorem \ref{injectivity}.
Each direct summand $\mathsf{IC} ( \eta ) \subset {}^p H ^i ( \mathcal E ^! )$ yields an isomorphism
$$\mathrm{Ext} ^{-i+m}_G ( \mathsf{IC} ( \eta ) [-i], \ddot{\mathcal L} ) \cong \begin{cases} P _{\eta, m} & (m \text{ is even}) \\ \{ 0 \} & (m \text{ is odd}) \end{cases}.$$
By taking $\mathrm{Hom}_G ( \bullet, \ddot{\mathcal L} )$, we obtain a (part of an) exact sequence
\begin{align}\nonumber
0 \to & \mathrm{Ext} ^{-i+2m}_G ( \tau _{> i} \mathcal E ^!, \ddot{\mathcal L} )\to \mathrm{Ext} ^{-i+2m}_G ( \tau _{\ge i} \mathcal E ^!, \ddot{\mathcal L} ) \to \mathrm{Ext} ^{-i+2m}_G ( {}^p H ^i ( \mathcal E ^! ) [-i], \ddot{\mathcal L} )\\
& \to \mathrm{Ext} ^{1-i+2m}_G ( \tau _{> i} \mathcal E ^!, \ddot{\mathcal L} ) \to \mathrm{Ext} ^{1-i+2m}_G ( \tau _{\ge i} \mathcal E ^!, \ddot{\mathcal L} ) \to 0\label{five-term}
\end{align}
for each $m \in \mathbb Z$. This exact sequence admits a weight filtration with respect to the Frobenius action (by utilizing $\phi$ and its induced linearizations).

For a mixed $G$-equivariant sheaf $\mathcal F_0$ (which is equivalent to $\mathcal F \in D^b _G ( \mathcal N )$ with a Frobenius linearization $\phi _{\mathcal F} : \mathsf{Fr} ^* \mathcal F \cong \mathcal F$), we denote $\mathtt{Gr} ^{\mathsf W} _k \mathrm{Ext} ^{m}_G ( \mathcal F, \ddot{\mathcal L} ) $ the weight $k$ part of $\mathrm{Ext} ^{m}_G ( \mathcal F, \ddot{\mathcal L} )$ for each $m, k \in \mathbb Z$ (after constructing its associated graded). Then, Claim \ref{comp-anal} implies that
$$\mathtt{Gr} ^{\mathsf W} _{- i + m + k} \mathrm{Ext} ^{- i + m}_G ( {}^p H ^{i} ( \mathcal E ^! ) [-i], \ddot{\mathcal L} ) = \{ 0 \} \hskip 2mm \text{ for all }i < 0, k \le 0\text{, and all }m \in \mathbb Z.$$ Applying this to (\ref{five-term}), we conclude that the sequence
\begin{align*}
& \mathtt{Gr} ^{\mathsf W} _{1-i+2m} \mathrm{Ext} ^{-i+2m}_G ( \tau _{\ge i} \mathcal E ^!, \ddot{\mathcal L} ) \to \mathtt{Gr} ^{\mathsf W} _{1-i+2m} \mathrm{Ext} ^{-i+2m}_G ( {}^p H ^i ( \mathcal E ^! )[-i], \ddot{\mathcal L} )\\
& \to \mathtt{Gr} ^{\mathsf W} _{1-i+2m} \mathrm{Ext} ^{1-i+2m}_G ( \tau _{> i} \mathcal E ^!, \ddot{\mathcal L} ) \to \mathtt{Gr} ^{\mathsf W} _{1-i+2m} \mathrm{Ext} ^{1-i+2m}_G ( \tau _{\ge i} \mathcal E ^!, \ddot{\mathcal L} ) \to 0
\end{align*}
must be exact and
$$\mathtt{Gr} ^{\mathsf W} _{-i+2m} \mathrm{Ext} ^{-i+2m}_G ( \tau _{> i} \mathcal E ^!, \ddot{\mathcal L} ) \cong \mathtt{Gr} ^{\mathsf W} _{-i+2m} \mathrm{Ext} ^{-i+2m}_G ( \tau _{\ge i} \mathcal E ^!, \ddot{\mathcal L} ) \hskip 2mm \text{ for all }m \in \mathbb Z.$$

In particular, if we write $\mathtt{Gr} ^{\mathsf W} _{i-1} {}^p H ^{i} ( \mathcal E ^{!} )_0$ by $\bigoplus _{\eta \in \Lambda} V_{\eta,-1} ^i \boxtimes \mathsf{IC} ( \eta )$, then the above short exact sequence turns into a short exact sequence
\begin{equation}
\bigoplus _{\eta \in \Lambda} V _{\eta, -1} ^i \boxtimes P _{\eta} \to \bigoplus _{m \ge 0} \mathtt{Gr} ^{\mathsf W} _{m} \mathrm{Ext} ^{m}_G ( \tau _{> i} \mathcal E ^!, \ddot{\mathcal L} ) \to \bigoplus _{m \ge 0} \mathtt{Gr} ^{\mathsf W} _{m} \mathrm{Ext} ^{m}_G ( \tau _{\ge i} \mathcal E ^!, \ddot{\mathcal L} ) \to 0\label{proj-kill}
\end{equation}
for each $i < 0$ and $m \in \mathbb Z$.

Thanks to the $A_W$-module structure of $\bigoplus _{m \ge 0} \mathtt{Gr} ^{\mathsf W} _{m} \mathrm{Ext} ^{m}_G ( \bullet, \ddot{\mathcal L} )$ arising from the Yoneda composition, we deduce the surjectivities of
$$P _{\chi} \longrightarrow \!\!\!\!\! \rightarrow \bigoplus _{m \ge 0} \mathtt{Gr} ^{\mathsf W} _{m} \mathrm{Ext} ^{m} _G ( \tau _{> i} \mathcal E ^!, \ddot{\mathcal L} ) \longrightarrow \!\!\!\!\! \rightarrow P_{\chi, \mathcal P}$$
for every $i \le -1$ by using (\ref{proj-kill}) repeatedly. Here the middle term is $P_{\chi}$ in the $i = -1$ case, while it is the pure-part of $H _{\bullet} ^{Z_G (x) ^{\circ}} ( \mathfrak B_x, \dot{\mathcal L}) _{\xi}$ in the $i \ll 0$ case. Since $H _{odd} ( \mathfrak B_x, \dot{\mathcal L}) _{\xi} = \{ 0 \}$ by Theorem \ref{LusSTD} 6), the Serre spectral sequence
$$E_2 (\chi) : = H ^{\bullet} _{Z_G (x) ^{\circ}} ( \mathrm{pt} ) \otimes H _{\bullet} ( \mathfrak B_x, \dot{\mathcal L}) \Rightarrow H _{\bullet} ^{Z_G (x) ^{\circ}} ( \mathfrak B_x, \dot{\mathcal L})$$
is $E_2$-degenerate. By Theorem \ref{LusSTD} 9), we conclude that $H _{\bullet} ^{Z_G (x) ^{\circ}} ( \mathfrak B_x, \dot{\mathcal L}) _{\xi}$ is pure. This implies that $H _{\bullet} ^{Z_G (x) ^{\circ}} ( \mathfrak B_x, \dot{\mathcal L})_{\xi}$ is a quotient of $P _{\chi}$. The $H ^{\bullet} _{Z_G (x) ^{\circ}} ( \mathrm{pt} )$-action commutes with the $A_W$-action (as the $H _{Z_G(x)}^{\bullet} ( \mathrm{pt} )$-module structure is obtained as a scalar extension of the $H _{G} ^{\bullet} ( \mathrm{pt} )$-module structure of $A_W$; cf. \cite{L-CG1} 8.13, 8.14). By the degeneracy of $E_2 (\chi)$, the forgetful map
$$\phi : H _{\bullet} ^{Z_G (x) ^{\circ}} ( \mathfrak B_x, \dot{\mathcal L}) _{\xi} \longrightarrow H _{\bullet} ( \mathfrak B_x, \dot{\mathcal L}) _{\xi} \cong K _{(x,\xi)} ^{\mathbf c, gen}$$
must be surjective. Thus, $\ker \, \phi$ is isomorphic to
\begin{equation*}
\Hom _{A_x} ( \xi, H ^{>0} _{Z_G (x) ^{\circ}} ( \{ x \} ) \otimes H _{\bullet} ( \mathfrak B_x, \dot{\mathcal L} ) _{\xi} \oplus \bigoplus _{\zeta \neq \xi} H ^{\bullet} _{Z_G (x) ^{\circ}} ( \{ x \} ) \otimes H _{\bullet} ( \mathfrak B_x, \dot{\mathcal L} ) _{\zeta} ).\label{annfor}
\end{equation*}
The surjectivity of $\phi$ implies that $H _{\bullet} ( \mathfrak B_x, \dot{\mathcal L}) _{\xi}$ is generated by its degree $0$-part. So are the same for every $\eta \in \mathsf{Irr} \, W$. Therefore, a generator set of $\ker \, \phi$ is contained in $H ^{\bullet} _{Z_G (x) ^{\circ}} ( \{ x \} ) \otimes H _{0} ( \mathfrak B_x, \dot{\mathcal L} )$. By Theorem \ref{LusSTD} 8) and 9), all the $W$-isotypic constituents of the latter space is of type $L_{\eta}$ with $\mathcal O _{\eta} = \mathcal O_{\chi}$. As a consequence, we have a sequence of surjective maps of graded $A_W$-modules
$$P _{\chi} \longrightarrow \!\!\!\!\! \rightarrow \mathrm{Ext} ^{\bullet}_G ( \mathcal E^!, \ddot{\mathcal L} ) \longrightarrow \!\!\!\!\! \rightarrow K _{(x,\xi)} ^{\mathbf c, gen} \longrightarrow \!\!\!\!\! \rightarrow P _{\chi, \mathcal P}.$$
In particular, $K _{(x,\xi)} ^{\mathbf c, gen}$ is a quotient of $P_{\chi}$. By Theorem \ref{LusSTD} 7), we deduce that $[ K _{(x,\xi)} ^{\mathbf c, gen} : L _{\chi'} ] \neq 0$ only if $\mathcal O _{\chi} \subset \overline{\mathcal O _{\chi'}} \setminus \mathcal O _{\chi'}$ or $\chi = \chi'$. Hence, $K _{(x,\xi)} ^{\mathbf c, gen}$ must be a quotient of $P _{\chi, \mathcal P}$. This implies $K _{(x,\xi)} ^{\mathbf c, gen} \cong P _{\chi, \mathcal P}$ as desired.
\end{proof}

\begin{definition}
Let $\mathbf c$ be a cuspidal datum. A phyla $\mathcal P$ is called an admissible phyla of $\mathbf c$ if each phylum is an equi-orbit class of the Springer correspondents with respect to $\mathbf c$ and a phylum has a smaller index if the dimension of an orbit is smaller.
\end{definition}

\begin{theorem}\label{gSp}
Let $\mathbf c$ be a cuspidal datum. For each ${\chi} \in \mathsf{Irr} \, W$ with its Springer correspondent $(x,\xi)$ $($with respect to $\mathbf{c})$, we define $K ^{\mathbf c} _{\chi} := K^{\mathbf c, gen} _{(x,\xi)}$.\\
Then, the collection $\{ K ^{\mathbf c} _{\chi} \} _{\chi \in \mathsf{Irr} \, W}$ gives rise to a Kostka system adapted to every admissible phyla $\mathcal P$ of ${\mathbf c}$.
\end{theorem}

\begin{proof}
By \cite{Lu2} 24.8b, the matrix $([ K ^{\mathbf c} _{\chi} : L _{\eta} ])$ satisfies (\ref{Smatrix}) for every refinement of the closure ordering. Hence, Theorem \ref{injectivity} implies that $\{ K ^{\mathbf c} _{\chi} \} _{\chi}$ is a Kostka system adapted to every admissible phyla $\mathcal P$ of ${\mathbf c}$ as required.
\end{proof}

\begin{corollary}\label{filt-eo}
Keep the setting of Theorem \ref{gSp}. For each $\chi \in \mathsf{Irr} \, W$, we define
$$\widetilde{K} _{\chi} := P_{\chi} / ( \sum _{\chi' < \chi, f \in \mathrm{hom} _{A_W} ( P_{\chi'}, P _{\chi} )} \mathrm{Im} f ),$$
where the ordering of $\mathsf{Irr} \, W$ is determined by an admissible phyla of $\mathbf c$. Then, $\widetilde{K} _{\chi}$ admits a separable decreasing $A_W$-module filtration whose successive quotients are of the form $\{ K ^{\mathbf c} _{\chi'} \} _{\chi' \sim \chi}$ up to grading shifts.
\end{corollary}

\begin{proof}
We employ the setting in the proof of Theorem \ref{injectivity}. The $A_W$-module $H _{\bullet} ^{Z_G(x)^{\circ}} ( \mathcal B_x ) _{\xi}$ is a quotient of $P_{\chi}$. It surjects onto $\widetilde{K} _{\chi}$ by a repeated use of (\ref{proj-kill}). Since the $H ^{\bullet} _{Z_G(x)^{\circ}} ( \mathrm{pt} )$-action commutes with the $W$-action, $H _{\bullet} ^{Z_G(x)^{\circ}} ( \mathcal B_x ) _{\xi}$ does not contain a $W$-type $L_{\chi'}$ with $\chi' < \chi$ by Theorem \ref{LusSTD} 7). Therefore, we have $H _{\bullet} ^{Z_G(x)^{\circ}} ( \mathcal B_x ) _{\xi} \cong \widetilde{K} _{\chi}$. For each $k \in \mathbb Z$, the subspace
$$\bigoplus _{\zeta \in \mathsf{Irr} A_x} \Hom _{A_x} ( \xi, H ^{\ge 2k} _{Z_G(x)^{\circ}} ( \mathrm{pt} ) \otimes H _{\bullet} ( \mathfrak B_x ) _{\zeta} ) \subset \widetilde{K} _{\chi}$$
is an $A_W$-submodule. Its associated graded is a direct sum of $A_W$-modules of the form $\{ H _{\bullet} ( \mathfrak B_x ) _{\zeta} \} _{\zeta}$ (up to grading shifts), and hence we conclude the result.
\end{proof}

\begin{corollary}\label{filt-gr}
Keep the setting of Corollary \ref{filt-eo}. Define $R_x := H ^{\bullet} _{Z_G(x)^{\circ}} ( \mathrm{pt} )$ to be the graded algebra equipped with an $A_x$-action. We have
$$\mathsf{gch} \, \widetilde{K} _{\chi} = \sum _{(x,\zeta) \sim (x,\xi)} \left( \mathsf{gdim} \, \Hom _{A_x} ( \xi \otimes \zeta ^{\vee}, R_x ) \right) \cdot \mathsf{gch} \, K ^{\mathbf c, gen} _{(x,\zeta)}.$$
In particular, we have $\widetilde{K} _{\chi} = K _{\chi} ^{\mathbf c}$ if $Z_G(x)^{\circ}$ is unipotent.
\end{corollary}

\begin{proof}
Compare the presentation of $\widetilde{K} _{\chi}$ in (\ref{gSpfiber}) and Corollary \ref{filt-eo}.
\end{proof}

\begin{corollary}\label{Lp}
We use the setting of Theorem \ref{gSp} and borrow the notation $\widetilde{K} _{\chi}$ and $R_x$ from Corollaries \ref{filt-eo} and \ref{filt-gr}. We define
$$\Xi _x := \{ \zeta \in \mathsf{Irr} A_x \!\mid ( x, \zeta ) \text{ is a Springer correspondent with respect to } \mathbf c \}.$$
We identify $\Xi_x$ with a subset of $\mathsf{Irr} \, W$. Form a graded algebra
\begin{align*}
& A_W^{\uparrow} := A_W / ( \sum _{\chi' < \chi} A_W e_{\chi'} A_W )
\text{ and set}\\
& R_x^{\mathbf c} := \bigoplus _{\xi,\zeta \in \Xi_x} \Hom _{A_x} ( \xi \otimes \zeta ^{\vee}, R_x), \hskip 3mm \mathtt K := \bigoplus _{\chi \in \Xi_x} \widetilde{K}_{\chi}.
\end{align*}
Then, we have an essentially surjective functor
$$A_W ^{\uparrow} \mathchar`-\mathsf{gmod} \ni M \mapsto \mathrm{hom}_{A_W} ( \mathtt K, M ) \in R_x ^{\mathbf c} \mathchar`-\mathsf{gmod}$$
which annihilates precisely the module which does not contain $L_{\chi}$ with $\chi \in \Xi_x$.
\end{corollary}

\begin{proof}
By construction, each $\widetilde{K} _{\chi}$ is a projective object in $A_W ^{\uparrow} \mathchar`-\mathsf{gmod}$. We have $\mathrm{hom} _{A_W} ( \mathtt{K}, L _{\chi'} ) = 0$ for every $\chi' > \Xi_x$. Thanks to Corollaries \ref{filt-eo} and \ref{filt-gr}, we deduce
$$\mathrm{hom}_{A_W} ( \mathtt{K}, \mathtt{K} ) \cong R_x ^{\mathbf c},$$
which is enough to see the assertion.
\end{proof}

\begin{corollary}\label{so}
Keep the setting of Theorem \ref{gSp}. We have:
\begin{enumerate}
\item $\mathrm{ext} ^{\bullet} _{A_W} ( \widetilde{K} _{\chi}, K ^{\mathbf c} _{\chi'} ) \neq \{ 0 \} \text{ only if } \chi > \chi' \text{ or } \chi = \chi'$;
\item $\mathrm{ext} ^{\bullet} _{A_W} ( K ^{\mathbf c} _{\chi}, K ^{\mathbf c} _{\chi'} ) \neq \{ 0 \} \text{ only if } \chi \gtrsim \chi'.$
\end{enumerate}
\end{corollary}

\begin{remark}
Corollary \ref{so} resembles the structure of the Ginzburg conjecture for affine Hecke algebras (\cite{G, B, TX, X}).
\end{remark}

\begin{proof}[Proof of Corollary \ref{so}]
Thanks to Corollaries \ref{filt-eo} and \ref{Lp}, {\bf 2)} follows from {\bf 1)}. 

We prove {\bf 1)}. Thanks to \cite{K3} 2.5, Claims \ref{j-ext} and \ref{comp-anal} imply that for each $a \in \mathbb Z$, we have a distinguished triangle:
$$\rightarrow \mathrm{gr} _a \, \mathcal E ^! \rightarrow F_{\ge a} \mathcal E ^! \rightarrow F _{> a} \mathcal E ^! \stackrel{+1}{\longrightarrow}$$
so that $F _{\ge a} \mathcal E ^! \cong \mathcal E ^!$ for $a \ll 0$, $\mathrm{gr} _a \, \mathcal E ^!$ is a mixed sheaf of pure weight $a$, $F_{\ge a} \mathcal E^!$ has weight $\ge a$, and $F _{> a} \mathcal E ^!$ has weight $> a$. In addition, each direct summand of $\mathrm{gr} _a \, \mathcal E ^!$ is isomorphic to a degree shift of $\{ \mathsf{IC} ( \chi' ) \} _{\chi' \in \Lambda}$ if $a < 0$, isomorphic to $\mathsf{IC} ( \chi )$ if $a = 0$, and $\{ 0 \}$ if $a > 0$.

For each $a \in \mathbb Z$, we set
$$Q _{a} ( \chi ) := \mathrm{Ext} ^{\bullet} _G ( \mathrm{gr} _a \, \mathcal E ^! , \ddot{\mathcal L} ).$$
This is a graded projective $A$-module. Each direct summand of $Q _{a} ( \chi )$ is a grading shifts of $P_{\chi'}$ $(\chi' \in \Lambda)$ for $a < 0$, and we have $Q_0 ( \chi ) \cong P _{\chi}$. In addition, the same argument as in \cite{K3} 2.7 and 2.8 (which are in turn applicable by Theorem \ref{LusSTD} 4) and 9), respectively) yields a projective resolution:
$$\rightarrow Q_{-2} ( \chi ) \stackrel{d_{-2}}{\longrightarrow} Q_{-1} ( \chi ) \stackrel{d_{-1}}{\longrightarrow}  Q_{0} ( \chi ) \stackrel{d_{0}}{\longrightarrow} \widetilde{K} _{\chi} \rightarrow 0.$$
This implies
$$\mathrm{ext} _{A} ^{\bullet} ( \widetilde{K} _{\chi}, L_{\eta} ) \neq \{ 0 \} \hskip 3mm \text{ only if } \hskip 3mm \eta \in \Lambda \hskip 3mm \text{ or } \hskip 3mm \eta = \chi.$$
Combined with Lemma \ref{invert}, we deduce {\bf 1)} as desired.
\end{proof}

\section{Lusztig-Slooten symbols of type $\mathsf{BC}$}\label{LSsymb}
We use the setting of \S \ref{genkos}. In this section, we consider the case $W = \mathfrak S _n \ltimes ( \mathbb Z / 2 \mathbb Z )^n$. Most of the assertions here are essentially not new. Nevertheless, we put explanations/proofs to each statement since we need to reinterpret them in order to make them fit into our framework.

Let $\Gamma := ( \mathbb Z / 2 \mathbb Z ) ^n \subset W$ denote the normal subgroup of $W$ so that $W = \mathfrak S _n \ltimes \Gamma$. Let $S_{\Gamma}$ be the set of reflections (of $W$) in $\Gamma$. We fix $\mathsf{Lsgn}$ (resp. $\mathsf{Ssgn}$) to be the one-dimensional representation of $W$ so that $\mathfrak S _n$ acts trivially and each element of $S_{\Gamma}$ acts by $-1$ (resp. $\mathfrak S _n$ acts by $\mathsf{sgn}$ and $\Gamma$ acts trivially).

For a bi-partition ${\bm \lambda} = ( \lambda ^{(0)}, \lambda ^{(1)} )$ of $n$, we define
$$W _{\bm \lambda} := \prod _{i \ge 1} \left(  W _{\lambda ^{(0)} _i} \times W _{\lambda ^{(1)} _i} \right) \subset W,$$
where $W _{k}$ is the Weyl group of type $\mathsf{BC} _k$. Let $\mathsf{mi}_{\bm \lambda}$ be the one-dimensional representation of $W _{\bm \lambda}$ on which $W _{\lambda ^{(0)} _i}$ acts by $\mathsf{Ssgn}$ and $W _{\lambda ^{(1)} _i}$ acts by $\mathsf{sgn}$. We also define $W^{\bm \lambda} := W_{|\lambda ^{(0)}|} \times W_{| \lambda ^{(1)}|} \subset W$.

\begin{fact}\label{typeCref}
There exists a bijection between $\mathsf{Irr} \, W$ and $\mathtt P ( n )$ so that:
\begin{enumerate}
\item For each partition $\lambda$, let $L _{\lambda}$ denote the $W$-representation obtained as the pullback by $W \to \!\!\!\!\! \to \mathfrak S _n$. For each ${\bm \lambda} = ( \lambda ^{(0)}, \lambda ^{(1)} ) \in \mathtt P ( n )$, we have
$$L _{\bm \lambda} \cong \mathsf{Ind} ^W _{W ^{\bm \lambda}} \left( ( L _{\lambda^{(0)}} \otimes \mathsf{Lsgn} ) \boxtimes L _{\lambda ^{(1)}} \right).$$
Exactly $|\lambda^{(0)}|$ elements of $S_{\Gamma}$ act by $-1$ on each $S_{\Gamma}$-eigenspace of $L_{\bm \lambda}$;
\item For each ${\bm \lambda} = ( \lambda ^{(0)}, \lambda ^{(1)} ) \in \mathtt P ( n )$, we have
$$\Hom _{W _{{}^{\mathtt t} {\bm \lambda}}} ( \mathsf{mi} _{{} ^{\mathtt t} {\bm \lambda}}, L _{\bm \lambda} ) \cong \mathbb C;$$
\item For each ${\bm \lambda} \in \mathtt P ( n )$, we have
$$\dim \mathrm{hom} _{A_W} ( P _{\bm \lambda}, P _{\mathsf{triv}}^*\left< 2 b ( {\bm \lambda} ) \right>) _{i} = \begin{cases}1 & (i=0)\\ 0 & (i > 0)\end{cases};$$
\item Let $K ^{ex} _{\bm \lambda}$ be the image of a non-zero map in {\bf 3)}. Then, we have
$$\dim \mathrm{hom} _{W} ( L _{\bm \mu}, K ^{ex} _{\bm \lambda} ) \neq 0 \text{ only if } b ( {\bm \lambda} ) \ge b ( {\bm \mu} ).$$
In addition, we have
$$\mathsf{gdim} \, \mathrm{hom} _{W} ( \mathsf{triv}, K ^{ex} _{\bm \lambda} ) = t ^{b ( {\bm \lambda} )} \text{ and } \mathsf{gdim} \, \mathrm{hom} _{W} (  L _{\bm \lambda}, K ^{ex} _{\bm \lambda} ) = 1;$$
\item For each ${\bm \lambda} = ( \lambda ^{(0)}, \lambda ^{(1)} ) \in \mathtt P ( n )$, we have
$$L _{{}^{\mathtt t} {\bm \lambda}} \cong L _{\bm \lambda} \otimes \mathsf{sgn} \text{ and }L _{( \lambda ^{(0)}, \lambda ^{(1)} )} \cong L _{( \lambda ^{(1)}, \lambda ^{(0)} )} \otimes \mathsf{Lsgn};$$
\item For each ${\bm \lambda} \in \mathtt P ( n )$, we have
$$\mathfrak h \otimes L _{\bm \lambda} \cong \bigoplus_{{\bm \lambda} \doteq {\bm \mu}}  L _{\bm \mu}.$$
\end{enumerate}
\end{fact}

\begin{proof}
{\bf 1)}--{\bf 5)} can be read-off from Carter \cite{Ca} \S 11. {\bf 6)} is Tokuyama \cite{To} Example 2.9.
\end{proof}

\begin{definition}[Symbols]
Let $r > 0$ and $s$ be real numbers. Fix an integer $m \gg n$ and form two sequences:
\begin{align*}
& r m \ge r (m-1) \ge \cdots \ge r \ge 0\\
& r m + s \ge r (m-1) + s \ge \cdots \ge r + s \ge s.
\end{align*}
We call this sequence ${\bm \Lambda} ^0$. For a bipartition $(\lambda^{(0)}, \lambda^{(1)})$ of $n$, we define a pair of two sequences ${\bm \Lambda} ( \lambda^{(0)}, \lambda^{(1)} )$ as:
\begin{align*}
& \lambda _1 ^{(0)} + r m \ge \lambda _2^{(0)} + r (m-1) \ge \cdots \ge \lambda _{m} ^{(0)} + r \ge 0\\
& \lambda _1 ^{(1)} + r m + s \ge \lambda _2 ^{(1)} + r (m-1) + s \ge \cdots \ge \lambda _{m} ^{(1)} + r + s \ge s.
\end{align*}
We call ${\bm \Lambda} ( \lambda^{(0)}, \lambda^{(1)} )$ the symbol (or the $(r,s)$-symbol) of a bi-partition $(\lambda^{(0)}, \lambda^{(1)})$. Let $Z ^{r,s}_n$ be the set of $(r,s)$-symbols obtained in this way (with $m$ fixed). We have a canonical identification $\Psi _{r,s} : \mathtt P (n) \stackrel{\cong}{\longrightarrow} Z ^{r,s}_n$, by which we identify bi-partitions with symbols.
\end{definition}

\begin{remark}
{\bf 1)} Adding $r$ uniformly to the sequences and add an additional last terms $0$ and $s$, we have a canonical identification of $Z ^{r,s}_n$ obtained by two different choices of $m$. We call this identification the shift equivalence. {\bf 2)} If we use ${\bm \Lambda} \in Z^{r,s}_n$ and ${\bm \Lambda} ^0 \in Z ^{r,s} _0$ simultaneously, then the value of $m$ is in common.
\end{remark}

\begin{definition}[$a$-functions, ordering, and similarity]\label{symBC}
For each ${\bm \Lambda} \in Z^{r,s}_n$, we consider ${\bm \Lambda} ^0 \in Z ^{r,s} _0$ and define
$$a ( {\bm \Lambda} ) = a _s ( {\bm \Lambda} ) := \sum _{a,b \in {\bm \Lambda}} \min ( a, b ) - \sum _{a,b \in {\bm \Lambda}^0} \min ( a, b ).$$
We might replace ${\bm \Lambda}$ with $\Psi _{r,s}^{-1} ( {\bm \Lambda} )$ if the meaning is clear from the context.\\
Two symbols ${\bm \Lambda}, {\bm \Lambda}' \in Z ^{r,s} _n$ are said to be similar if the entries of ${\bm \Lambda}$ and ${\bm \Lambda}'$ are in common (counted with multiplicities), and we denote it by ${\bm \Lambda} \sim {\bm \Lambda}'$. They are said to be strongly similar if ${\bm \Lambda}'$ is obtained from ${\bm \Lambda}$ by swapping several pairs of type $(k, k+1)$ or $(k+1, k)$ (for some $k \in \mathbb Z$) from the first and second sequences, and we denote it by ${\bm \Lambda} \approx {\bm \Lambda}'$.\\
For ${\bm \Lambda}, {\bm \Lambda}' \in Z^{r,s}_n$, we define ${\bm \Lambda} > {\bm \Lambda}'$ if $a ( {\bm \Lambda} ) < a ( {\bm \Lambda}' )$. We refer this partial ordering as the $a$-function ordering. We define a phylum associated to $Z^{r,s} _n$ as a similarity class, and a phyla associated to $Z ^{r,s} _n$ as the set of all similarity classes, ordered in an arbitrary compatible way as the $a$-function ordering.
\end{definition}

\begin{remark}
It is easy to see that the similarity classes and the strong similarity classes of $Z ^{r,s} _n$ are independent of the choice of $m$, and the $a$-function depends only on the similarity class. In particular, the $a$-function does not depend on the choice of $m \gg n$ (cf. Shoji \cite{Sh2} 1.2).
\end{remark}

In the below, we assume $r = 2$ as in \cite{L-IC, Sl} unless otherwise stated.

\begin{lemma}[Lusztig \cite{L-IC}, Slooten \cite{Sl}]\label{gSpdata}
Let $s, n \in \mathbb Z _{>0}$. If $s$ is odd, then the similarity classes and the $a$-function of $Z_n^{2,s}$ coincide with the orbits and the half of the orbit codimensions $($inside the subvariety $\mathcal N \subset \mathcal N _G$ defined in \S 3$)$ of a generalized Springer correspondence of a symplectic group.\\
Similarly, if $s \equiv 2 \mod 4$, then they coincide with those of a generalized Springer correspondence of an odd orthogonal group. If $s \equiv 0 \mod 4$, then the same is true for an even orthogonal group.
\end{lemma}

\begin{remark}
{\bf 1)} Thanks to Lemma \ref{gSpdata}, a phyla associated to $Z^{2,s} _n$ (for $s \in \mathbb Z_{> 0}$) is an admissible phyla of a generalized Springer correspondence. {\bf 2)} In the symbol notation, swapping the first and second sequences correspond to tensoring $\mathsf{Lsgn}$, which gives an equivalent but different system. The $W$-module structure we employ are those coming from the $\mathsf{sgn}$-twists of irreducible tempered modules of affine Hecke algebras as in \cite{L-CG3,Sl,CK,CKK}.
\end{remark}

\begin{proof}[Proof of Lemma \ref{gSpdata}]
By rearranging $m$ if necessary, we can assume that the last $s$-entries of each sequence of ${\bm \lambda} \in Z ^{2,s}_n$ does not have effect neither on a similarity class nor the $a$-function. Then, the bijection of \cite{L-IC} (12.2.2)--(12.2.3) can be seen as setting $s := 1 - 2 d$, where $d$ is the defect of the symbols ({\it loc. cit.} P256L-8). Here $d$ is a priori an odd integer, and hence we realize $s \equiv 1 \mod 4$. For $s \equiv 3 \mod 4$, we can swap the role of the first and second sequences whenever $d > 0$ to deduce the symbol combinatorics on similarity classes. This, together with {\it loc. cit.} Corollary 12.4c, implies that a similarity class of $Z ^{2,s} _n$ is the same as an equi-orbit class of some generalized Springer correspondence of symplectic groups. Since the constant local system on a nilpotent orbit gives rise to a Springer representation (original one, $d=1, s=-1$ case), we conclude that the $a$-function on $Z^{2,s} _n$ calculate the half of the codimensions of orbits again by {\it loc. cit.} 12.4c and the normalization condition $a _s ( \emptyset, ( n ) ) = 0$ for $s > 0$. The case of even $s$ is similar ({\it loc.cit.} \S 13).
\end{proof}

\begin{corollary}\label{refine}
Keep the setting of Lemma \ref{gSpdata}. For each positive integer $s$, every phyla associated to $Z _n ^{2,s}$ gives rise to the same solution of $(\ref{Smatrix})$.
\end{corollary}

\begin{proof}
A direct consequence of Theorem \ref{gSp} and Lemma \ref{gSpdata}.
\end{proof}

In the below, if the (complete collection of) $\mathcal P$-traces $\mathsf{P} = \{ P _{{\bm \lambda}, \mathcal P} \} _{\bm \lambda \in {\mathtt P} ( n )}$ with respect to a phyla associated to $Z^{r,s} _n$ also gives the set of $\mathcal P$-traces with respect to {\it every} phyla associated to $Z^{r,s}_n$, then we call $\mathsf{P}$ the set of $\mathcal P$-traces adapted to $Z^{r,s}_n$.

In particular, we refer a Kostka system $\mathsf{K}$ adapted to every phyla associated to $Z ^{r,s} _n$ as a Kostka system adapted to $Z^{r,s} _n$. We denote by $\{ K ^s _{\bm \lambda} \} _{\bm \lambda \in \mathtt{P} ( n )}$ the Kostka system adapted to $Z ^{2,s} _n$ for each $s \in \mathbb Z _{>0}$ (which exists by Theorem \ref{gSp}).

\begin{lemma}[Slooten \cite{Sl}]\label{generic}
For $s \not\in \mathbb Z$, a phyla associated to $Z ^{2,s} _n$ is singleton.
\end{lemma}

\begin{proof}
An entry of the first row of a symbol of $Z ^{2,s}_n$ is always an integer, while an entry of the second row of a symbol of $Z ^{2,s} _n$ is always not an integer. Hence, they cannot mix up.
\end{proof}

\begin{proposition}[Slooten \cite{Sl} 4.2.8]\label{Wind}
Let $s \in \mathbb Z _{\ge 0}$. Let ${\bm \lambda} = ( \lambda ^{(0)}, \lambda^{(1)}) \in \mathtt P ( n - k )$ for some integer $k$. We define
\begin{align*}
X _s ( k, {\bm \lambda} ) & := \{ {\bm \mu} \in \mathtt P ( n ) \mid [ \mathsf{Ind} ^ W _{\mathfrak S_k \times W_{n-k}} ( \mathsf{triv} \boxtimes L _{\bm \lambda} ) : L _{\bm \mu}] \neq 0 \}\\
Y _s ( k, {\bm \lambda} ) & := \{ {\bm \mu} \in X _s ( k, {\bm \lambda} ) \mid a _s ( {\bm \mu}) \ge a _s ( {\bm \gamma} ) \text{ for every } {\bm \gamma} \in X _s ( k, {\bm \lambda} ) \}.
\end{align*}
Then, ${\bm \mu} = ( \mu^{(0)}, \mu^{(1)}) \in Y _s ( k, {\bm \lambda} )$ satisfies:
\begin{itemize}
\item There exists a subdivision $k = k _0 + k _1$ so that we have $\{ \mu ^{(j)}_i \}_i = \{ \lambda ^{(j)}_i \}_i \cup \{ k_j \}$ for $j=0,1$, where we allow repetitions in the both sets;
\item We can choose $p, q$ so that $\mu ^{(0)} _p = k_0$, $\mu^{(1)} _q = k_1$, and
$$k_0 + 2 q - s = k_1 + 2 p \pm 1 \text{ or } k_1 + 2 p.$$
\end{itemize}
In addition, the set $Y _s ( k, {\bm \lambda} )$ is either a singleton or a pair of strongly similar symbols of $Z^{2,s} _n$.
\end{proposition}

\begin{proof}
This is exactly the same as \cite{Sl} 4.2.8. For the compatibility with our choice of symbols, see \cite{Sl} 4.5.2.
\end{proof}

\begin{lemma}[Slooten \cite{Sl} \S 4.5]\label{char-seq}
For each strong similarity class $\mathcal S$ of $Z ^{2,s} _n$, we have a set $E ( \mathcal S )$ of entries of ${\bm \Lambda} \in \mathcal S$ with the following properties: 
\begin{itemize}
\item The assignment
$$\mathcal S \ni {\bm \Lambda} \mapsto \sigma ^s _{\bm \Lambda} := ( E ( \mathcal S ) \cap \{ \text{entries of the second row of } {\bm \Lambda} \} ) \in 2 ^{E ( \mathcal S )}$$
sets up a bijection between $\mathcal S$ and $2 ^{E ( \mathcal S )}$;
\item For ${\bm \Lambda}, {\bm \Lambda}' \in \mathcal S$, we have $a _{s + \epsilon} ( {\bm \Lambda} ) > a _{s + \epsilon} ( {\bm \Lambda}' )$ if $\sigma ^s _{\bm \Lambda} \supset \sigma ^s _{{\bm \Lambda}'}$;
\item For ${\bm \Lambda}, {\bm \Lambda}' \in \mathcal S$, we have $a _{s - \epsilon} ( {\bm \Lambda} ) > a _{s - \epsilon} ( {\bm \Lambda}' )$ if $\sigma ^s _{\bm \Lambda} \subset \sigma ^s _{{\bm \Lambda}'}$.
\end{itemize}
Here $0< \epsilon \ll 1$ is a real number.
\end{lemma}

\begin{proof}
Each sequence of a symbol cannot contain a consecutive sequence of integers (since $r = 2$). Let $I = \{ p,p+1,\ldots, q\}$ be a consecutive sequence of integers appearing in ${\bm \Lambda}$ so that $(p-1), (q+1) \not\in {\bm \Lambda}$. Then, its division $I^+ := \{p, p + 2,\ldots\}$ and $I^- := \{p+1, p+3,\ldots\}$ must belong to distinct sequences. If $\# I \ge 2$, then none of the element of $I$ appears twice in ${\bm \Lambda}$. Hence, we can swap $I^+$ and $I^-$ simultaneously (if $\# I^+ = \# I ^-$), but not individually. Therefore, a symbol is characterized (inside its strong similarity class) by the behaviour of such sequences with even length. As a consequence, the set $E ( \mathcal S )$ consisting of minimal entries ($p$ in the above) of such sequences $I$ satisfies the first assertion. We write $q_p$ the length of the sequence $I \ni p \in E ( \mathcal S)$. Then, for each ${\bm \Lambda}, {\bm \Lambda}' \in \mathcal S$ and $|\kappa| \ll 1$, we have
$$a _{s + \kappa} ( {\bm \Lambda} ) - a _{s + \kappa} ( {\bm \Lambda}' ) = \kappa ( \sum _{p \in \sigma ^s _{{\bm \Lambda}}} q_p - \sum _{p' \in \sigma ^s _{{\bm \Lambda}'}} q_{p'} )$$
by inspection. This is enough to prove the other two assertions.
\end{proof}

\begin{theorem}[Slooten \cite{Sl2}, Ciubotaru-K \cite{CK, CKK}]\label{delimits}
For each $s \in \mathbb Z_{> 0}$ and $0 < \epsilon < 1$, we have a collection $\{ K ^{s + \epsilon} _{\bm \lambda} \} _{{\bm \lambda} \in {\mathtt P} ( n )}$ of indecomposable $A_W$-modules with the following properties:
\begin{enumerate}
\item The module $K ^{s + \epsilon} _{\bm \lambda}$ is a quotient of $P_{\bm \lambda}$, and we have $[K ^{s + \epsilon} _{\bm \lambda} : L _{\bm \lambda}] = 1$;
\item Let $\mathcal S \subset Z^{2,s}_n$ be the strong similarity class which contains ${\bm \lambda}$. We have
$$\mathsf{gch} \, K^{s+\epsilon} _{\bm \lambda} \equiv \sum _{{\bm \gamma} \in \mathcal S, \, \sigma ^s _{\bm \gamma} \subset \sigma ^s _{\bm \lambda}} \mathsf{gch} \, K ^{s} _{\bm \gamma} \mod ( t - 1 );$$
\item Let $\mathcal S \subset Z^{2,s+1}_n$ be the strong similarity class which contains ${\bm \lambda}$. We have
$$\mathsf{gch} \, K^{s+\epsilon} _{\bm \lambda} \equiv \sum _{{\bm \gamma} \in \mathcal S, \, \sigma ^{s+1} _{\bm \gamma} \supset \sigma ^{s+1} _{\bm \lambda}} \mathsf{gch} \, K ^{s+1} _{\bm \gamma} \mod ( t - 1 ).$$
\end{enumerate}
\end{theorem} 

\begin{proof}
First, we observe that the integer $s$ corresponds to the graded Hecke algebra parameter ratio $s / 2$ by Lemma \ref{gSpdata} (and its proof) and Lusztig \cite{L-CG1} 2.13 (cf. \cite{Sl} 3.6.1). We have the set of (isomorphism classes of) irreducible tempered modules $\{ M _{\bm \lambda} ^{s + \epsilon} \} _{\bm \lambda}$ of a graded Hecke algebra $\mathcal H$ of type $\mathsf{BC}$ (see \cite{CK} \S 1.2 for the definition) with real central characters whose parameter ratio is $(s+\epsilon) / 2$. The set $\{ M _{\bm \lambda} ^{s + \epsilon} \} _{\bm \lambda}$ is known to be in bijection with the set of irreducible representations of $W$ by Lusztig \cite{L-CG3} 1.21 (cf. \cite{L-CG2} 10.13 and \cite{L-IC}) when $\epsilon \in \{ 0, \frac{1}{2}, 1 \}$, and by \cite{CK} Theorem C and \S 4.3 for $0 < \epsilon < 1$.

Thanks to Opdam \cite{Op} and Slooten \cite{Sl2} (cf. \cite{CK} Theorem C), we know that $M _{\bm \lambda} ^{s + \epsilon}$ is written as a unique irreducible induction from a discrete series representation. In addition, its $W$-module structure is
\begin{equation}
M _{\bm \lambda} ^{s + \epsilon} \cong \mathsf{Ind} ^W _{( \mathfrak S _{\lambda^{\mathsf{A}}} \times W_{(n-k)} )} \mathbb C \boxtimes M ^{s + \epsilon} _{{\bm \lambda} ^{\mathsf{C}}},\label{genT}
\end{equation}
where $\lambda^{\mathsf{A}}$ is a partition of $k$, ${\bm \lambda} ^{\mathsf{C}}$ is a bi-partition of $(n-k)$, and $M^{s+\epsilon} _{{\bm \lambda} ^{\mathsf{C}}}$ is a discrete series representation of graded Hecke algebra $\mathcal H'$ of type $\mathsf{BC}$ with the same parameter ratio $(s+\epsilon)/2$, but has rank $(n-k)$.

\begin{claim}[Slooten \cite{Sl}]\label{Slabel}
The module $L_{\bm \lambda}$ in $(\ref{genT})$ is the $W$-irreducible constituent of $\mathsf{Ind} ^W _{( \mathfrak S _{\lambda^{\mathsf{A}}} \times W_{(n-k)} )} \mathbb C \boxtimes L _{{\bm \lambda} ^{\mathsf{C}}}$ whose label attains the maximal $a_{s+\epsilon}$-function value $($which is in fact unique$)$. Moreover, it defines a unique bijection between the set of tempered modules of $\mathcal H$ with real central characters and $\mathsf{Irr} \, W$ so that $L_{\bm \lambda} \subset M _{\bm \lambda} ^{s + \epsilon}$ $($as $W$-modules$)$.
\end{claim}

\begin{proof}
The first assertion is established in Slooten (\cite{Sl} 4.5.6) up to the property $L_{\bm \lambda} \subset M _{\bm \lambda} ^{s + \epsilon}$. By construction, it is enough to check it for discrete series. This is given in \cite{CK} \S 4.4 as the matching of Lusztig's $W$-types (of a generalized Springer correspondence of a Spin group) and Slooten's combinatorics.

In addition, \cite{CK} \S 4.5 and \cite{L-CG3} shows that the $W$-characters of $\{ M _{\bm \lambda} ^{s + \epsilon} \} _{\bm \lambda}$ is equal to those of $\{ K^{\mathbf c} _{\bm \lambda}\} _{\bm \lambda}$ for some cuspidal datum $\mathbf c$. Thanks to the triangularity condition of the matrix $K$ in the Lusztig-Shoji algorithm (Theorem \ref{Sh}), we deduce that a bijection in the assertion must be unique as required. 
\end{proof}

We return to the proof of Theorem \ref{delimits}. Thanks to \cite{CKK} 3.16, each $M^{s+\epsilon} _{{\bm \lambda} ^{\mathsf{C}}}$ is isomorphic to (two) irreducible tempered modules of $\mathcal H'$ with their parameter ratios $s/2$ and $(s+1)/2$ as $W$-modules. By \cite{L-CG3} 1.17, 1.21, 1.22 (and Theorem \ref{gSp}), we identify $\{K^{s} _{\bm \lambda}\} _{\bm \lambda}$ with the set of irreducible tempered modules (viewed as $W$-modules) with real central characters of $\mathcal H$ with its parameter ratio $s/2$. By utilizing \cite{CKK} 3.15, 3.25 (cf. \cite{Sl2} 3.5.3), we deduce that the ungraded $W$-character
$$\mathsf{ch} \, M _{\bm \lambda} ^{s + \epsilon} \in \mathbb Z \mathsf{Irr} \, W \subset \mathbb Z (\!( t )\!) \mathsf{Irr} \, W$$
satisfies
\begin{equation}
\mathsf{ch} \, M _{\bm \lambda} ^{s + \epsilon} \equiv \sum _{{\bm \mu} \in \mathcal T _{\bm \lambda}} \mathsf{gch} \, K^{s} _{\bm \mu} \mod ( t - 1 )\label{delim}
\end{equation}
for some set $\mathcal T _{\bm \lambda} \subset \mathtt{P} ( n )$. Put $\mathcal S _{\bm \lambda} := \{ {\bm \mu} \in \mathcal S \!\mid \sigma ^s _{\bm \mu} \subset \sigma ^s _{\bm \lambda} \}$. By the comparison of \cite{CKK} 3.15, 3.22 with Proposition \ref{Wind}, Lemma \ref{char-seq} (cf. \cite{Sl2} 3.4.4), we obtain a bijection $\mathcal S _{\bm \lambda} \cong \mathcal T _{\bm \lambda}$ so that $\mathcal S _{\bm \lambda} \subset \mathcal S _{\bm \lambda'}$ implies $\mathcal T _{\bm \lambda} \subset \mathcal T _{\bm \lambda'}$ for each ${\bm \lambda'} \in \mathcal S$. In view of \cite{CKK} 3.24 and 3.25, the bijections $\mathcal S \cong \mathcal T$ yield a bijection $\varphi : \mathtt{P} (n) \cong \mathtt{P} (n)$ so that $\varphi ( {\bm \lambda}) \in \mathcal T _{\bm \lambda}$ and
$$L _{\varphi ( {\bm \lambda} )} \subset K ^{s} _{\varphi ( {\bm \lambda} )} \subset M ^{s+\epsilon} _{{\bm \lambda}} \hskip 5mm \text{ as $W$-modules}$$
for each ${\bm \lambda} \in \mathtt{P} (n)$. By the uniqueness part of Claim \ref{Slabel}, we deduce $\varphi = \mathrm{id}$. In particular, we conclude $\mathcal T _{\bm \lambda} = \mathcal S _{\bm \lambda}$. Thanks to \cite{Lu} 4.13, Proposition \ref{fl} 1)--3) is satisfied. Applying Proposition \ref{fl}, we obtain a collection of modules $\{ K^{s+\epsilon} _{\bm \lambda} \} _{\bm \lambda}$ which satisfies the condition {\bf 1)}, and $\mathsf{gch} \, K^{s+\epsilon} _{\bm \lambda} \equiv \mathsf{ch} \, M _{\bm \lambda} ^{s + \epsilon} \mod ( t - 1 )$. Combined with (\ref{delim}), we deduce the condition {\bf 2)}.

The condition {\bf 3)} follows from a similar argument as above by replacing $\{ K ^{s} _{\bm \lambda} \} _{\bm \lambda}$ with $\{ K ^{s+1} _{\bm \lambda} \} _{\bm \lambda}$, identified with the set of irreducible tempered modules of a graded Hecke algebra of type $\mathsf{BC}$ whose parameter ratio is $(s+1)/2$.
\end{proof}

\begin{corollary}\label{spin}
Keep the setting of Theorem \ref{delimits}. The collection $\{ K ^{s + \epsilon} _{\bm \lambda} \} _{{\bm \lambda} \in {\mathtt P} ( n )}$ is a Kostka system adapted to an admissible phyla of a generalized Springer correspondence of a $\mathrm{Spin}$-group.
\end{corollary}

\begin{proof}
By the proof of Theorem \ref{delimits}, $\{ K ^{s + \epsilon} _{\bm \lambda} \} _{{\bm \lambda}}$ is isomorphic to the Kostka system in the assertion as a set of $W$-modules. Since each $K ^{s + \epsilon} _{\bm \lambda}$ is a quotient of $P _{\bm \lambda}$, we conclude the isomorphism as a set of graded $A_W$-modules by the $\mathcal P$-trace characterization of Kostka systems (Definition \ref{kos} {\bf 1)}).
\end{proof}

\section{Transition of Kostka systems in type $\mathsf{BC}$}\label{transBC}

Keep the setting of the previous section.

\begin{lemma}\label{filt}
Let $s \in \mathbb Z _{>0}$ and $0 < \epsilon < 1$. For each strong similarity class $\mathcal S \subset Z ^{2,s}_n$ and ${\bm \lambda} \in \mathcal S$, the $A_W$-module $K ^{s+\epsilon} _{\bm \lambda}$ $($borrowed from Theorem \ref{delimits}$)$ admits a filtration whose successive quotients are of the form $\{ K ^s _{\bm \mu} \} _{\bm \mu \in \mathcal S}$ up to grading shifts. If $s > 1$, then $K ^{(s-1)+\epsilon} _{\bm \lambda}$ also admits a filtration whose successive quotients are of the form $\{ K ^{s} _{\bm \mu} \} _{\bm \mu \in \mathcal S}$ up to grading shifts.
\end{lemma}

\begin{proof}
Since the proofs of the both cases are essentially the same, we prove only the first half of the assertion. By Theorem \ref{delimits} 2), we deduce that
\begin{equation}
[ K _{\bm \lambda}^{s+\epsilon} : L _{\bm \mu} ] \MID _{t=1} = 1 \hskip 3mm \text{(}{\bm \mu} \in \mathcal S \text{ and } \sigma ^s _{\bm \mu} \subset \sigma ^s _{\bm \lambda} \text{) , and} \hskip 3mm 0 \hskip 3mm \text{(otherwise)}\label{ceq}
\end{equation}
for each ${\bm \mu} \in \mathtt P ( n )$ such that $a _{s} ( {\bm \mu} ) \ge a _{s} ( {\bm \lambda} )$. We set $M^0 := \{ 0 \} \subset K^{s+\epsilon} _{\bm \lambda}$. Then, by assuming the existence of the submodule $M^{i-1}$, we construct an $A_W$-submodule $M^i$ of $K^{s+\epsilon} _{\bm \lambda}$ which is spanned by $M^{i-1}$ and a unique $L _{\bm \mu}$ with $a _s ( {\bm \mu} ) = a _s ( {\bm \lambda} )$ such that $M^i / M^{i-1}$ contains no other irreducible $W$-constituent of type $L _{{\bm \gamma}}$ with $a _{s} ( {\bm \gamma} ) = a _{s} ( {\bm \lambda} )$. Each $M^i / M^{i-1}$ is a quotient of $K^s _{{\bm \mu}}$ with ${\bm \mu}$ coming from (\ref{ceq}) since $K^s _{{\bm \mu}}$ is a $\mathcal P$-trace adapted to $Z ^{2,s} _n$. Hence, we have
\begin{equation}
\dim K ^{s+\epsilon} _{\bm \lambda} = \sum _{i \ge 1} \dim M^i / M^{i-1} \le \sum _{{\bm \mu} \in \mathcal S, \, \sigma ^s _{\bm \mu} \subset \sigma ^s _{\bm \lambda}} \dim K ^s _{\bm \mu}.\label{deq}
\end{equation}
The most RHS of (\ref{deq}) is equal to $\dim K ^{s+\epsilon} _{\bm \lambda}$ again by Theorem \ref{delimits} 2). Therefore, conclude that $M^i / M^{i-1} \cong K^s _{{\bm \lambda}_i} \left< d _i \right>$ for some $d _i \in \mathbb Z _{\ge 0}$ and ${\bm \lambda}_i \in \mathcal S$ such that $\sigma ^s _{{\bm \lambda}_i} \subset \sigma ^s _{\bm \lambda}$. This implies that $K^{s+\epsilon} _{\bm \lambda}$ admits an $A_W$-module filtration whose successive quotients are $\{ K^s _{\bm \lambda} \} _{\bm \lambda}$ as required.
\end{proof}

\begin{lemma}\label{trace-hom}
We fix $s \in \mathbb Z_{>0}$ and $0 < \epsilon \ll 1$. Let $\mathcal S$ be a strong similarity class of $Z^{2,s}_n$, and let $\{ P _{\bm \lambda, \star} \} _{\bm \lambda}$ be the collection of $\mathcal P$-traces with respect to $Z ^{2,s+\epsilon} _n$. For ${\bm \lambda}, {\bm \mu} \in \mathcal S$ such that $\sigma ^s _{\bm \lambda} \subsetneq \sigma ^s _{{\bm \mu}}$, we have:
\begin{equation}
\dim \hom_{A_W} ( P _{\bm \lambda} \left< 2 d _{{\bm \lambda}, {\bm \mu}} \right>, P _{{\bm \mu}, \star} )_0 \ge 1.\label{estPtr>}
\end{equation}
The same assertion holds for $\mathcal P$-traces with respect to $Z^{2,s-\epsilon}_n$ if we assume $\sigma ^s _{{\bm \mu}} \subsetneq \sigma ^s _{{\bm \lambda}}$.
\end{lemma}

\begin{proof}
Since the proofs of the both cases are similar, we prove the assertion only for the $\mathcal P$-traces with respect to $Z^{2,s+\epsilon}_n$. We set $d := d _{{\bm \lambda}, {\bm \mu}}$.

By the proof of Lemma \ref{char-seq}, we know that ${\bm \lambda}$ is obtained from ${\bm \mu} = ( \mu^{(0)}, \mu^{(1)} )$ by swapping $( \# \sigma ^s _{{\bm \mu}} - \# \sigma ^s _{\bm \lambda} )$ entries of ${}^{\mathtt t} ( \mu ^{(0)} )$ with those of ${}^{\mathtt t} ( \mu^{(1)} )$. (Here we rephrased symbol combinatorics by bi-partition combinatorics.) In particular, we have a bi-partition ${\bm \delta} = ( \delta ^{(0)}, \delta ^{(1)}) \in \mathtt P ( n - d )$ so that $\delta ^{(0)} = \lambda ^{(0)}$ and $\delta ^{(1)} = \mu ^{(1)}$. Moreover, there exists a partition $\kappa$ of $d$ so that $( {} ^{\mathtt t} \lambda ^{(1)} ) _{j_i} = ( {} ^{\mathtt t} \delta ^{(1)} ) _{j_i} + ( {}^{\mathtt t} \kappa ) _{i}$ and $( {} ^{\mathtt t} \mu ^{(0)} ) _{j'_i} = ( {} ^{\mathtt t} \delta ^{(0)} ) _{j'_i} + ( {}^{\mathtt t} \kappa ) _{i}$ for some sequences $\{ j_i \}$ and $\{j'_i\}$.

\begin{claim}\label{injwithrest}
The inequality $(\ref{estPtr>})$ is true if we have $L _{\bm \lambda} \subset S^d \mathfrak h \otimes L _{\bm \mu}$.
\end{claim}
\begin{proof}
Every sequence of bi-partitions
$${\bm \mu} = {\bm \lambda} _0 \doteq {\bm \lambda} _1 \doteq \cdots \doteq {\bm \lambda} _{d} = {\bm \lambda} \hskip 2mm \text{ with } \hskip 2mm {\bm \lambda} _i = ( \lambda _i ^{(0)}, \lambda_i^{(1)})$$
satisfies $| \lambda _i ^{(0)} | = | \mu ^{(0)} | - i$ for each $0 \le i \le d$. In addition, every such sequence must satisfy inequalities
$$a _{s+ \epsilon} ( {\bm \mu} ) > a _{s + \epsilon} ( {\bm \lambda}_i ) \hskip 3mm \text{ for every } \hskip 3mm i > 0$$
by inspection. Thanks to Fact \ref{typeCref} 6), it follows that any non-zero map in $\hom_{A_W} ( P _{\bm \lambda} \left< 2 d \right>, P _{{\bm \mu}} )_0$ gives rise to a non-zero map in $\hom_{A_W} ( P _{\bm \lambda} \left< 2 d \right>, P _{{\bm \mu, \star}} )_0$. Thus, $L _{\bm \lambda} \subset S^d \mathfrak h \otimes L _{\bm \mu}$ is enough to prove (\ref{estPtr>}).
\end{proof}

We return to the proof of Lemma \ref{trace-hom}.

Recall that the Frobenius reciprocity (and Fact \ref{typeCref} 1)) asserts
\begin{align}\nonumber
& \Hom _{W _{|\mu^{(0)}|}} ( L _{( \delta^{(0)}, 1^d )}, S^d \mathfrak h \otimes L _{( \mu^{(0)}, \emptyset)} )\\
& \cong \Hom _{( W _{|\delta^{(0)}|} \times W_d )} ( L _{( \delta^{(0)}, \emptyset )} \boxtimes L _{( \emptyset, 1^d )}, S^d \mathfrak h \otimes L _{( \mu ^{(0)}, \emptyset )} ).\label{primLR}
\end{align}
Applying the Littlewood-Richardson rule (Macdonald \cite{Mac} I \S 9, applied in the sign-twisted form; cf. Fact \ref{typeAref} 4)) and the Frobenius reciprocity, we deduce
$$L _{( \mu^{(0)}, \emptyset )} \MID _{( W _{|\delta^{(0)}|} \times W _d )} \supset L _{(\delta^{(0)}, \emptyset)} \boxtimes L _{(1^d,\emptyset)},$$
which is in fact a multiplicity-free copy. Let $\mathfrak h' \subset \mathfrak h$ be the reflection representation of $W_d$. Notice that $\wedge_+ ^d \mathfrak h$ is the sum of $S_{\Gamma}$-eigenspaces of $S^d \mathfrak h$ so that exactly $d$ elements of $S_{\Gamma}$ act by $-1$. We have $\wedge^d_+  \mathfrak h ' \subset \wedge^d_+ \mathfrak h \subset S ^d \mathfrak h$ as $W_d$-modules. In addition, we have $\wedge^d_+ \mathfrak h' \cong \mathsf{Lsgn}$ as a $W_d$-module. It follows that
\begin{equation}
L _{( \delta^{(0)}, \emptyset )} \boxtimes L _{(\emptyset,1^d)} \subset \wedge^d_+ \mathfrak h \otimes L _{( \mu ^{(0)}, \emptyset )} \subset S ^d \mathfrak h \otimes L _{( \mu ^{(0)}, \emptyset )}\label{mi-ex}
\end{equation}
as $W_{|\delta^{(0)}|} \times W_d$-modules. Therefore, we deduce
\begin{align}\nonumber
S ^d \mathfrak h \otimes L _{\bm \mu} & \supset \mathsf{Ind} ^W _{( W _{|\mu^{(0)}|} \times W _{|\mu^{(1)}|} )} ( S^d \mathfrak h \otimes L _{(\mu^{(0)}, \emptyset)} ) \boxtimes L _{(\emptyset, \mu^{(1)})}\\
& \supset \mathsf{Ind} ^W _{( W _{|\delta^{(0)}|} \times W_d \times W _{|\delta^{(1)}|} )} L _{(\delta^{(0)}, \emptyset)} \boxtimes L _{( \emptyset, 1^d )} \boxtimes L _{(\emptyset, \delta ^{(1)})} \supset L _{\bm \lambda},\label{mt+}
\end{align}
where the first inclusion is by adjunction, the second inclusion is (\ref{mi-ex}), and the last one is the Littlewood-Richardson rule. This completes the proof.
\end{proof}

\begin{lemma}\label{dist}
Let $s \in \mathbb Z _{>0}$ and $0 < \epsilon < 1$. Assume that $\{ K ^{s+\epsilon} _{\bm \lambda} \} _{\bm \lambda}$ is a Kostka system adapted to $Z^{2, s + \epsilon} _n$. Then, we have
\begin{equation}
\mathsf{gch} \, K ^{s+\epsilon} _{\bm \lambda} = \sum _{\Psi _{2,s} ( {\bm \mu} ) \approx \Psi _{2,s} ( {\bm \lambda} ), \, \sigma ^s _{\bm \mu} \subset \sigma ^s _{\bm \lambda}} t ^{d_{{\bm \lambda}, {\bm \mu}}} \mathsf{gch} \, K ^s _{\bm \mu}.\label{char-eq+}
\end{equation}
Similarly, if $\{ K ^{s+\epsilon} _{\bm \lambda} \} _{\bm \lambda}$ is a Kostka system adapted to $Z^{2, s + 1 - \epsilon} _n$, then we have
\begin{equation*}
\mathsf{gch} \, K ^{s+\epsilon} _{\bm \lambda} = \sum _{\Psi _{2,(s+1)} ( {\bm \mu} ) \approx \Psi _{2,(s+1)} ( {\bm \lambda} ), \, \sigma ^{s+1} _{\bm \mu} \supset \sigma ^{s+1} _{\bm \lambda}} t ^{d_{{\bm \lambda}, {\bm \mu}}} \mathsf{gch} \, K ^{s+1} _{\bm \mu}.
\end{equation*}
\end{lemma}

\begin{proof}
Since the proofs of the both assertions are completely parallel, we prove only the first assertion. Recall (from Theorem \ref{delimits}) that
$$[ K _{\bm \lambda}^{s+\epsilon} : L _{\bm \mu} ] \MID _{t=1} = 1 \hskip 3mm \text{(} \Psi _{2,s} ( {\bm \mu} ) \approx \Psi _{2,s} ( {\bm \lambda} ) \text{ and } \sigma ^s _{\bm \mu} \subset \sigma ^s _{\bm \lambda} \text{) , and} \hskip 3mm 0 \hskip 3mm \text{(otherwise)}$$
for each ${\bm \mu} \in \mathtt P ( n )$ so that $\Psi _{2,s} ( {\bm \mu} ) \sim \Psi _{2,s} ( {\bm \lambda} )$. Applying Lemma \ref{trace-hom}, we conclude $[ K _{\bm \lambda}^{s+\epsilon} : L _{\bm \mu} ] = t^{d_{{\bm \lambda}, {\bm \mu}}}$ if it is nonzero. This, together with Lemma \ref{filt}, implies
\begin{equation*}
\mathsf{gch} \, K ^{s+\epsilon} _{\bm \lambda} = \sum _{\Psi _{2,s} ( {\bm \mu} ) \approx \Psi _{2,s} ( {\bm \lambda} ), \, \sigma ^s _{\bm \mu} \subset \sigma ^s _{\bm \lambda}} t ^{d_{{\bm \lambda}, {\bm \mu}}} \mathsf{gch} \, K ^s _{\bm \mu}
\end{equation*}
as desired.
\end{proof}

\begin{proposition}\label{ak}
We take an arbitrary $r \in \mathbb Z _{> 0}$. Let $s \gg 0$. For a bi-partition ${\bm \lambda} = (\lambda ^{(0)}, \lambda^{(1)})$, we define $A ^{\bm \lambda} := A _{W, W^{\bm \lambda}} = \mathbb C W^{\bm \lambda} \ltimes \mathbb C [\mathfrak h^*] \subset A _W$. If we put
$$K _{\bm \lambda} := A_W \otimes _{A^{\bm \lambda}} \left( K ^{ex} _{( \lambda ^{(0)}, \emptyset )} \boxtimes L _{( \emptyset, \lambda ^{(1)})} \right),$$
then $\{ K _{\bm \lambda} \} _{{\bm \lambda} \in \mathtt P ( n )}$ gives rise to a Kostka system adapted to $Z ^{r,s} _n$.
\end{proposition}

\begin{proof}
Postponed to Appendix B.
\end{proof}

\begin{theorem}\label{Kmain} For each $s' \in \mathbb R_{\ge 1}$, there exist a Kostka system adapted to $Z^{2,s'}_n$. In addition, we have:
\begin{itemize}
\item Fix $s \in \mathbb Z_{>0}$. For $0 < \epsilon < 1$, the Kostka system adapted to $Z^{2,s + \epsilon}_n$ do not depend on the choice of $\epsilon$. We denote them by $\{ K _{\bm \lambda} ^{\circ} \} _{{\bm \lambda}}$;
\item The Kostka system $\{ K^s _{\bm \lambda} \} _{{\bm \lambda}}$ adapted to $Z^{2,s} _n$ or the Kostka system $\{ K^{s+1} _{\bm \lambda} \} _{{\bm \lambda}}$ adapted to $Z^{2,s+1} _n$ determine the graded characters of the Kostka system $\{ K _{\bm \lambda} ^{\circ} \} _{{\bm \lambda}}$ as follows:
\begin{enumerate}
\item For a strong similarity class $\mathcal S \subset Z ^{2,s} _n$ and ${\bm \lambda} \in \mathcal S$, we have
$$\mathsf{gch} \, K _{\bm \lambda} ^{\circ} = \sum _{{\bm \mu} \in \mathcal S, \, \sigma ^s _{\bm \mu} \subset \sigma ^s _{\bm \lambda}} t ^{d _{{\bm \lambda}, {\bm \mu}}} \mathsf{gch} \, K^s _{{\bm \mu}};$$
\item For a strong similarity class $\mathcal S \subset Z ^{2,s+1} _n$ and ${\bm \lambda} \in \mathcal S$, we have
$$\mathsf{gch} \, K _{\bm \lambda} ^{\circ} = \sum _{{\bm \mu} \in \mathcal S, \, \sigma ^{s+1} _{{\bm \mu}} \supset \sigma ^{s+1} _{\bm \lambda}} t ^{d _{{\bm \lambda}, {\bm \mu}}} \mathsf{gch} \, K^s _{{\bm \mu}}.$$
\end{enumerate}
\end{itemize} 
\end{theorem}

\begin{proof}
The first assertion holds if $s' \in \mathbb Z_{> 0}$. Fix $s \in \mathbb Z _{>0}$ so that $s \le s' \le s+1$.

We borrow some notation from Theorem \ref{delimits}. An admissible phyla of the generalized Springer correspondence attached to a cuspidal datum $\mathbf c$ (of a Spin group) is singleton (i.e. at most one local system on each orbit contributes as a Springer correspondent; \cite{L-IC} 14.4--14.5). Therefore, Corollary \ref{spin} implies
\begin{align}\label{singlespin}
& \left< K _{\bm \lambda} ^{s + \epsilon}, ( K _{{\bm \mu}} ^{s + \epsilon} ) ^* \right>_{\mathsf{gEP}} = 0, \text{ and either } \\\nonumber
& \mathrm{ext} ^1 _{A_W} ( K^{s + \epsilon} _{\bm \lambda}, L_{{\bm \mu}} ) = \{ 0 \} \text{ and } [K^{s + \epsilon} _{\bm \mu} : L_{{\bm \lambda}}] = 0, \text{ or }\\\nonumber
& \mathrm{ext} ^1 _{A_W} ( K^{s + \epsilon} _{\bm \mu}, L_{{\bm \lambda}} ) = \{ 0 \} \text{ and } [K^{s + \epsilon} _{\bm \lambda}: L_{{\bm \mu}}] = 0
\end{align}
if ${\bm \lambda} \neq {\bm \mu}$. Thanks to (the both cases of) Lemma \ref{filt}, we deduce
\begin{equation}
\mathrm{ext} ^1 _{A_W} ( K^{s + \epsilon} _{\bm \lambda}, L_{{\bm \mu}} ) = \{ 0 \} \text{ and } [K^{s + \epsilon} _{\bm \mu} : L_{{\bm \lambda}}] = 0\label{kos-em}
\end{equation}
if either $a_s ( {\bm \lambda} ) > a_s ( {\bm \mu} )$ or $a_{s+1} ( {\bm \lambda} ) > a_{s+1} ( {\bm \mu} )$ holds. As each $a _{s + \epsilon} ( {\bm \lambda} )$ is linear with respect to $0 \le \epsilon \le 1$, we conclude that (\ref{kos-em}) holds if $a_{s + \epsilon} ( {\bm \lambda} ) > a_{s+\epsilon} ( {\bm \mu} )$ for some $0 < \epsilon < 1$.

\begin{claim}\label{eq-str}
Let ${\bm \lambda}, {\bm \mu} \in \mathtt{P} ( n )$ be a pair so that $a _{s+\epsilon} ( \bm \lambda ) = a_{s+ \epsilon} ( \bm \mu )$ for all $0 \le \epsilon \le 1$. Then, we have either $\Psi _{2,s} ( {\bm \lambda} ) \not\sim \Psi _{2,s} ( {\bm \mu} )$ or $\Psi _{2,(s+1)} ( {\bm \lambda} ) \not\sim \Psi _{2,(s+1)} ( {\bm \mu} )$.
\end{claim}

\begin{proof}
If $\Psi _{2,s} ( {\bm \lambda} ) \sim \Psi _{2,s} ( {\bm \mu} )$, then there exists a multiplicity-free entry $f$ in $\Psi _{2,s} ( {\bm \lambda} )$ so that $f$ belongs to the first sequence of $\Psi _{2,s} ( {\bm \lambda} )$, and also belongs to the second sequence of $\Psi _{2,s} ( {\bm \mu} )$. Then, $\Psi _{2,(s+1)} ( {\bm \lambda} )$ must contain $f$ as its entry, while $\Psi _{2,(s+1)} ( {\bm \mu} )$ cannot. Thus, we conclude $\Psi _{2,(s+1)} ( {\bm \lambda} ) \not\sim \Psi _{2,(s+1)} ( {\bm \mu} )$ as required.
\end{proof}

We return to the proof of Theorem \ref{Kmain}. Thanks to (the both cases of) Lemma \ref{filt}, we conclude that for each $0 < \epsilon < 1$, we have
\begin{equation}
[ K^{s + \epsilon} _{\bm \lambda} : L_{{\bm \mu}} ] = \delta _{{\bm \lambda}, {\bm \mu}} \hskip 2mm \text{ if } \hskip 2mm a_{s + \epsilon} ( {\bm \lambda} ) \le a_{s + \epsilon} ( {\bm \mu} ).\label{mult-van}
\end{equation}
Let $\mathcal P_{s+\epsilon}$ be the phlya defined as follows: Each phylum is of the form $a _{s+\epsilon} ^{-1} ( \alpha )$ for some $\alpha \in \mathbb R$. We have $a _{s+\epsilon} ^{-1} ( \alpha ) < _{\mathcal P_{s+\epsilon}} a _{s+\epsilon} ^{-1} ( \beta )$ if and only if $\alpha > \beta \in \mathbb R$.

By (\ref{mult-van}) and (\ref{kos-em}), we deduce that $\{ K^{s + \epsilon} _{\bm \lambda} \} _{\bm \lambda}$ is the set of $\mathcal P_{s+\epsilon}$-traces. Therefore, Proposition \ref{ext} and (\ref{singlespin}) implies 
$$\mathrm{ext} ^1 _{A_W} ( K^{s + \epsilon} _{\bm \lambda}, L_{{\bm \mu}} ) = \{ 0 \} \hskip 2mm \text{ if } \hskip 2mm {\bm \lambda} \neq {\bm \mu} \hskip 2mm \text{ and } \hskip 2mm a_{s + \epsilon} ( {\bm \lambda} ) \ge a_{s + \epsilon} ( {\bm \mu} ).$$
Now Corollary \ref{oc} and (\ref{singlespin}) implies that setting $K^{\circ} _{\bm \lambda} := K^{s + \epsilon} _{\bm \lambda}$ (which does not depend on $0< \epsilon < 1$ by Theorem \ref{delimits}) yields a Kostka system adapted to $Z ^{2,s+\epsilon}_n$. This proves the first two assertion. The last assertion follow from Lemma \ref{dist}.
\end{proof}

\begin{remark}[on Theorem \ref{Kmain}]
Since distances and the strong similarity classes are easily computable, the knowledge of $\{\mathsf{gch} \, K ^{\circ} _{\bm \lambda}\} _{\bm \lambda}$ is enough to determine the other two, namely $\{\mathsf{gch} \, K ^{s} _{\bm \lambda}\} _{\bm \lambda}$ and $\{\mathsf{gch} \, K ^{s+1} _{\bm \lambda}\} _{\bm \lambda}$. Combined with Proposition \ref{ak} (and Lemma \ref{gind}), we can compute $\{ \mathsf{gch} \, K _{\bm \lambda} ^{s'} \} _{\bm \lambda}$ for every $s' \in \mathbb R_{\ge 1}$ by Kostka polynomials of type $\mathsf{A}$ and the Littlewood-Richardson rules.
\end{remark}

\begin{corollary}\label{omid}
Keep the setting of Theorem \ref{Kmain}. The Kostka system $\{ K ^{\circ} _{\bm \lambda} \} _{\bm \lambda}$ satisfies
$$\mathrm{ext} ^{\bullet} _{A_W} ( K ^{\circ} _{\bm \lambda}, K ^{\circ} _{\bm \mu} ) \neq \{ 0 \} \text{ only if } {\bm \mu} \lesssim {\bm \lambda},$$
where the ordering is determined by a phyla associated to $Z_n ^{2, s + \epsilon}$.
\end{corollary}

\begin{proof}
If $a _{s} ( \bm \lambda ) > a_{s} ( \bm \mu )$ or $a _{s+1} ( \bm \lambda ) > a_{s+1} ( \bm \mu )$, then we appeal to Corollary \ref{so} 2) and Lemma \ref{filt} to deduce the assertion. We assume $a _{s+\epsilon} ( \bm \lambda ) = a_{s+ \epsilon} ( \bm \mu )$ for all $0 \le \epsilon \le 1$. For each pair ${\bm \lambda}, {\bm \mu} \in \mathtt{P} (n)$ so that ${\bm \lambda} \not\sim {\bm \mu}$ in $Z_n^{2,s}$ (i.e. $\Psi _{2,s} ( {\bm \lambda} ) \not\sim \Psi _{2,s} ( {\bm \mu} )$), we have
$$\mathrm{ext} ^{\bullet} ( K ^{s} _{\bm \lambda}, K^s _{\bm \mu} ) = \{ 0 \} \text{ and } \mathrm{ext} ^{\bullet} ( K ^{s} _{\bm \mu}, K^{s} _{\bm \lambda} ) = \{ 0 \}$$
by Theorem \ref{so}, which proves the assertion in this case. The same is true if we replace $s$ with $s+1$. This completes the proof by Claim \ref{eq-str}.
\end{proof}

\section*{Appendix A: Kostka systems in symmetric groups}
\renewcommand{\thesection}{\Alph{section}}
\setcounter{section}{1}
\setcounter{theorem}{0}
\setcounter{equation}{0}
In this appendix, we consider the case $W = \mathfrak S _n$. We present a Kostka system adapted to its natural ordering without relying on Theorem \ref{gSp}, that depends on geometric considerations. We employ the setting of \S \ref{genkos}.

\begin{fact}\label{typeAref}
In the same notation as in \S \ref{Wnot}, we have:
\begin{enumerate}
\item For a partition $\lambda$, we have
$$\dim \mathrm{hom} _{A_W} ( P _{\lambda}, P _{(n)}^* \left< 2 a ( \lambda ) \right> ) _{0} = 1.$$
Let $M _{\lambda}$ be the image of this unique homomorphism $($up to a scalar$)$. It gives rise to a solution $\{ [ M _{\lambda} : L _{\mu}] \} _{\lambda, \mu}$ of the equation $(\ref{Smatrix})$ corresponding to every total refinement of the ordering from \S \ref{Wnot};
\item As $\mathfrak{S} _n$-modules, we have an isomorphism
$$M _\lambda \cong \mathsf{Ind} ^{\mathfrak S _n} _{\mathfrak S_{\lambda}} \mathsf{triv};$$
\item We have $L _{{}^{\mathtt t} \lambda} \cong L _{\lambda} \otimes \mathsf{sgn}$, and $M _{\lambda} \otimes \mathsf{sgn} \cong \mathsf{Ind} ^{\mathfrak S _n} _{\mathfrak S_{\lambda}} \mathsf{sgn}$;
\item For two partitions $\lambda, \mu$ of $n$, we have $\lambda \ge \mu$ if and only if ${}^{\mathtt t} \lambda \le {}^{\mathtt t} \mu$.
\end{enumerate}
\end{fact}

\begin{proof}
{\bf 1)} and {\bf 2)} are reformulation of De Concini-Procesi \cite{DP} obtained by dualizing the quotient map $\mathbb C [ \mathfrak h^* ] \cong P _{(n)} \to M _{\lambda} ^* \left< 2 a ( \lambda ) \right>$. {\bf 3)} and {\bf 4)} can be read-off from Carter \cite{Ca} \S 11, together with the Frobenius reciprocity.
\end{proof}

\begin{remark}
There is an alternate combinatorial proof of Fact \ref{typeAref} 1) and 2) by Garsia-Procesi \cite{GP}. Thus, the proof of Theorem \ref{kosA} gives rise to a part of an algebraic proof of the whole story.
\end{remark}

\begin{corollary}\label{topsocle}
For each partition $\lambda$, the $A_W$-module $M_{\lambda}$ has simple head $L_{\lambda}$ and simple socle $\mathsf{triv} \left< 2 a ( \lambda ) \right>$. \hfill $\Box$
\end{corollary}

\begin{theorem}\label{kosA}
The collection $\{ M_{\lambda} \} _{\lambda}$ satisfies
$$\mathrm{ext} ^i _{A_{\mathfrak S_n}} ( M_{\lambda}, L _{\mu} ) = \{ 0 \} \hskip 3mm \text{ for every }  \mu \not\le \lambda \text{ and } i = 0,1.$$
In particular, $\{ M_{\lambda} \} _{\lambda}$ is a Kostka system.
\end{theorem}

The rest of this section is devoted to the proof of Theorem \ref{kosA}. By Corollary \ref{topsocle}, it suffices to prove $i= 1$ case.

We have an inclusion
$$M_{\lambda} \supset M _{\lambda, 0} = L _{\lambda} \supset \mathsf{sgn} \text{ as } \mathfrak S _{{}^{\mathtt t} \lambda} \text{-modules}.$$
We set $M_{\lambda} ^{\downarrow} := A_{\mathfrak S_n, \mathfrak S _{{}^{\mathtt t} \lambda}} \cdot \mathsf{sgn} \subset M _{\lambda}$. We name this embedding $\psi$. Since $M _{\lambda}$ is a submodule of $P _{\mathsf{triv}}^* \left< 2 a ( \lambda ) \right>$, we conclude that the $\mathbb C [ \mathfrak h^* ]$-action on
$$M_{\lambda} \subset P _{\mathsf{triv}}^* \left< 2 a ( \lambda ) \right> \cong \mathbb C [ \mathfrak h ] \left< 2 a ( \lambda ) \right>$$
is given by differentials. Consider the external tensor product factorization $A_{\mathfrak S_n, \mathfrak S_{{} ^{\mathtt t} \lambda}} \cong \boxtimes _i A _{\mathfrak S _{( {}^{\mathtt t} \lambda ) _i}}$ of graded algebras. The $\mathfrak S _{( {}^{\mathtt t} \lambda ) _i}$-module $\mathsf{sgn}$ yields an $A _{\mathfrak S _{( {}^{\mathtt t} \lambda ) _i}}$-module $P_{\mathsf{sgn}_i} ^{(0)} = P_{\mathsf{sgn}} / \bigl< J _{\mathfrak S _{( {}^{\mathtt t} \lambda ) _i}} \bigr> P_{\mathsf{sgn}}$, and its projective cover $P_{\mathsf{sgn}_i}$. The graded $A_{\mathfrak S_n, \mathfrak S_{{} ^{\mathtt t} \lambda}}$-module $M _{\lambda} ^{\downarrow}$ admits the corresponding factorization:
$$M _{\lambda} ^{\downarrow} \cong \boxtimes_{i=1}^{\lambda_1} P_{\mathsf{sgn} _i} ^{(0)} \subset \mathbb C [ \mathfrak h ] \left< 2 a ( \lambda ) \right>.$$
It follows that the minimal projective resolution of $M _{\lambda} ^{\downarrow}$ (as $A_{\mathfrak S_n, \mathfrak S_{{} ^{\mathtt t} \lambda}}$-modules) involves only the grading shifts of $\boxtimes_i P_{\mathsf{sgn} _i}$.

We have $M_{\lambda, 0} = L _{\lambda} = \sum _{w \in \mathfrak S _n} w\, \psi ( M _{\lambda,0} ^{\downarrow} )$ by the irreducibility of $L_{\lambda}$. It follows that $M_{\lambda} = \sum _{w \in \mathfrak S_n} w\, \psi (M_{\lambda} ^{\downarrow})$ by the top-term generation property of $M_{\lambda}$. Every non-trivial extension of $M_{\lambda}$ by $L_{\mu} \left< d \right>$ induces a non-trivial extension as $\mathbb C [\mathfrak h^*]$-modules by the semi-simplicity of $\mathbb C \mathfrak S_n$.

Assume that we have a non-split short exact sequence
\begin{equation}
0 \to L _{\mu} \left< d \right> \longrightarrow E \longrightarrow M_{\lambda} \to 0\label{shortE}
\end{equation}
of $A_{\mathfrak S_n}$-modules. We choose a $\mathbb C$-spanning set $e_1,\ldots, e_{k}$ of $E_{d-2} = M_{\lambda,(d-2)}$. Then, we have $\{0\} \neq \sum _{i=1} ^k \mathfrak h e _i \cap L _{\mu} \left< d \right> \subset E_d$ by the non-split assumption. It follows that for some $w \in \mathfrak S_n$, the short exact sequence (\ref{shortE}) induces a non-splitting short exact sequence
$$0 \to L _{\mu} \left< d \right> \longrightarrow E' \longrightarrow w\, \psi ( M_{\lambda}^{\downarrow} ) \to 0$$
of $\mathbb C [ \mathfrak h ^* ]$-modules.

By twisting by $w^{-1}$ if necessary, we can assume $w = \text{id}$ without the loss of generality. This makes us possible to view the above exact sequence as that of $A_{\mathfrak S_n, \mathfrak S _{{}^{\mathtt t} \lambda}}$-modules. As $M_{\lambda}^{\downarrow}$ admits a projective resolution consisting of grading shifts of $\boxtimes_i P_{\mathsf{sgn} _i}$ as $A_{\mathfrak S_n, \mathfrak S_{{}^{\mathtt t} \lambda}}$-modules, it follows that its extension by a simple graded $A_{\mathfrak S_n, \mathfrak S _{{}^{\mathtt t} \lambda}}$-module $L$ is non-zero if and only if $L \cong \mathsf{sgn} \left< d \right>$ for some $d$ as a $\mathfrak S _{{}^{\mathtt t} \lambda}$-module. Hence we need $\mathsf{sgn} \subset L _{\mu} \MID _{\mathfrak S _{{}^{\mathtt t} \lambda}}$ to satisfy the non-split assumption on (\ref{shortE}). By Fact \ref{typeAref} 3) and 2), we deduce that
$$\{ 0 \} \neq \Hom _{\mathfrak S _{{}^{\mathtt t} \lambda}} ( \mathsf{sgn}, L_{\mu}) \cong \Hom _{\mathfrak S _{{}^{\mathtt t} \lambda}} ( \mathsf{triv}, L_{{}^{\mathtt t} \mu} ) \cong \Hom _{\mathfrak S _n} ( M_{{}^{\mathtt t} \lambda}, L_{{}^{\mathtt t} \mu} ).$$
By Fact \ref{typeAref} 1), this implies ${}^{\mathtt t} \lambda \le {}^{\mathtt t} \mu$. Therefore, we have $\lambda \ge \mu$ by Fact \ref{typeAref} 4). This means that
\begin{eqnarray}
\mathrm{ext} ^1 _{A_{\mathfrak S _n}} ( M_{\lambda}, L _{\mu} ) \neq \{ 0 \} \hskip 3mm \text{ only if } \hskip 3mm \mu \le \lambda, \label{P-tr-A'}
\end{eqnarray}
which is equivalent to the first part of the assertion.

\section*{Appendix B: Asymptotic type $\mathsf{BC}$ case}\label{akpf}
\setcounter{section}{2}
\setcounter{theorem}{0}
\setcounter{equation}{0}

We employ the same setting as in \S \ref{LSsymb} and borrow some notation from Appendix A. This appendix is devoted to the proof of Proposition \ref{ak}.

\begin{lemma}
Let $\lambda$ and $\mu$ be distinct partitions of $n$. We have
$$\mathrm{ext} ^{\bullet} _{A_W}( L_{(\emptyset, \lambda)}, L _{(\emptyset, \mu)} ) = \{ 0 \}.$$
\end{lemma}

\begin{proof}
Observe that we have a Koszul resolution $\{ \wedge _+ ^{k} \otimes P_{(\emptyset, \lambda)} \left< 2k \right> \} _{k=0}^n$ of $L_{(\emptyset, \lambda)}$. By Fact \ref{typeCref} 1) and 6), we deduce that an irreducible $W$-constituent of $\wedge ^k _+ \otimes L _{(\emptyset, \lambda)}$ is of the form $L _{(\emptyset, \gamma)}$ for a partition $\gamma$ if and only if $k=0$ and $\gamma = \lambda$. It follows that every indecomposable summand of $\bigoplus _{k>0} \wedge _+ ^{k} \otimes P_{(\emptyset, \lambda)} \left< 2 k \right>$ is not of the form $P_{(\emptyset, \gamma)} \left< l \right>$ for a partition $\gamma$ and $l \in \mathbb Z$. Therefore, we conclude the result.
\end{proof}

\begin{lemma}\label{aasymp}
Let $r \in \mathbb Z _{> 0}$ and $s \gg 0$. Let ${\bm \lambda} = (\lambda ^{(0)}, \lambda^{(1)})$ and ${\bm \mu} = (\mu ^{(0)}, \mu^{(1)})$ be two bi-partitions of $n$ regarded as elements of $Z ^{r,s} _n$. Suppose that we have one of the followings:
\begin{eqnarray*}
|\lambda^{(0)}| > |\mu ^{(0)}|, \text{ or } a ( \lambda ^{(0)}) > a ( \mu ^{(0)}) \text{ and } \lambda ^{(1)} = \mu ^{(1)}.
\end{eqnarray*}
Then, we have $a ( {\bm \lambda} ) > a ( {\bm \mu} )$.
\end{lemma}

\begin{proof}
Notice that each element of $\lambda ^{(0)}$ contributes more than $n$-times, while each element of $\lambda ^{(1)}$ contributes less than or equal to $(n-1)$-times. Therefore, if $m \gg n + s/r \gg n$, then the first case follows. The other case is immediate.
\end{proof}

\begin{lemma}\label{gind}
We define $A ^{\flat} := \mathbb C W \ltimes \mathbb C [\epsilon_1 ^2, \ldots, \epsilon _n^2] \subset A_W$. Consider the natural degree-doubling embedding $A _{\mathfrak S _n} \subset A ^{\flat}$ and regard $M _{\lambda}$ as an $A ^{\flat}$-module by letting $\Gamma$ act trivially. Then we have
$$K^{ex} _{( \lambda, \emptyset )} \otimes \mathsf{Lsgn} \cong A_W \otimes _{A^{\flat}} M _{\lambda}$$
for each partition $\lambda$ of $n$.
\end{lemma}

\begin{proof}
The algebra $A_W$ is a free $A^{\flat}$-module with its free basis
\begin{equation}
1, \epsilon_1,\epsilon_2, \ldots, \epsilon_n,\epsilon _1 \epsilon _2, \epsilon_1 \epsilon_3, \ldots, \epsilon _1 \epsilon _2 \cdots \epsilon_n.\label{wbasis}
\end{equation}
It follows that the induction functor $A_W \otimes _{A ^{\flat}} \bullet$ preserves projective objects, and preserves the indecomposability. The indecomposable $A_W$-module $A_W \otimes _{A ^{\flat}} \mathsf{triv}$ has simple socle $\mathsf{Lsgn} \left< 2 n \right>$. Hence, we apply the induction functor to Fact \ref{typeAref} 3),4) to obtain a non-zero degree $0$ morphism
$$P _{( \emptyset, \lambda )} \to P ^* _{\mathsf{Lsgn}} \left< 4 a ( \lambda ) + 2 n \right>.$$
By twisting $\mathsf{Lsgn}$ to the both sides and applying Fact \ref{typeCref} 2) with an identity $2 a ( \lambda ) + n = b ( \lambda, \emptyset )$, we conclude the result.
\end{proof}

\begin{corollary}\label{asymp}
The module $K^{ex} _{(\lambda, \emptyset )}$ admits a graded projective resolution by using only $\{ P_{( \mu, \emptyset )} \left< d \right> \} _{\mu, d}$'s.
\end{corollary}

\begin{proof}
The induction functor $A_W \otimes _{A^{\flat}} \bullet$ sends an indecomposable module $P _{\lambda}$ to $P _{( \emptyset, \lambda )}$. Hence, $K ^{ex} _{(\lambda, \emptyset)} \otimes \mathsf{Lsgn}$ admits a graded projective resolution by using only $\{ P_{( \emptyset, \mu )} \left< d \right> \} _{\mu, d}$'s. By twisting $\mathsf{Lsgn}$ as in Lemma \ref{gind}, we conclude the assertion.
\end{proof}

\begin{lemma}\label{vcc}
For two distinct partitions $\lambda, \mu$ of $n$, we have
$$\left< K ^{ex} _{( \lambda, \emptyset )}, ( K ^{ex} _{( \mu, \emptyset )} )^* \right> _{\mathsf{gEP}} = 0.$$
Assume that Corollary \ref{so} holds for type $\mathsf{A}$. Then, we have
$$\mathrm{ext} ^{\bullet} _{A_W} ( K ^{ex} _{( \lambda, \emptyset )}, L _{( \mu, \emptyset )} ) = \{ 0 \} \hskip 3mm \text{ for each } \hskip 3mm \mu \not\le \lambda.$$
\end{lemma}

\begin{proof}
By the arguments in the proof of Corollary \ref{asymp}, if
$$P_i := \bigoplus _{\gamma,d \ge 2i} P_{\gamma} \left< d \right> ^{\oplus m _{\gamma,d}^i}$$
is the $i$-th term of the minimal projective resolution of $M_{\lambda}$, then
$$P_i^{\uparrow} := \bigoplus _{\gamma,d \ge 2i} P_{(\gamma, \emptyset)} \left< 2 d \right> ^{\oplus m _{\gamma,d}^i} = A_W \otimes _{A^{\flat}} P _i \otimes \mathsf{Lsgn}$$
is the $i$-th term of a projective resolution of $K^{ex} _{(\lambda,\emptyset)}$. It follows that if we write $\left< M_{\lambda}, L_{\mu} \right>_{\mathsf{gEP}} = Q_{\lambda,\mu} (t)$, then we have
$$\left< K^{ex}_{( \lambda, \emptyset )}, L_{( \mu, \emptyset )} \right>_{\mathsf{gEP}} = Q_{\lambda,\mu} ( t^2 ).$$
Thus, we conclude the desired vanishing of the graded Euler-Poincar\'e pairing by Theorem \ref{kosA} (or Theorem \ref{gSp}). For the second assertion, we have
$$\dim \mathrm{ext} ^i _{A_W} ( K ^{ex} _{( \lambda, \emptyset )}, L _{( \mu, \emptyset )} ) \le \dim \mathrm{ext} ^i _{A^{\flat}} ( M _{\lambda}, L _{\mu} ) \hskip 5mm \text{ for each } \hskip 2mm i \in \mathbb Z$$
by the above description of a projective resolution. Therefore, the assertion follows by Corollary \ref{so} (for type $\mathsf{A}$).
\end{proof}

We return to the proof of Proposition \ref{ak}. Let $n_i := | \lambda^{(i)}|$ for $i = 0,1$. Let $\mathfrak h_i \subset \mathfrak h$ be the reflection representation of $W_{n_i}$. We have $A ^{\bm \lambda} \cong (\mathbb C W_{n_0} \ltimes \mathbb C [ \mathfrak h ^* _0 ] ) \boxtimes ( \mathbb C W_{n_1} \ltimes \mathbb C [ \mathfrak h ^* _1 ] )$. We have
$$\mathrm{ext} ^{i} _{A_W} ( K _{\bm \lambda}, L _{\bm \mu} ) = \mathrm{ext} ^{i} _{A ^{\bm \lambda}} ( K ^{ex} _{( \lambda ^{(0)}, \emptyset )} \boxtimes L _{( \emptyset, \lambda ^{(1)} )}, L _{\bm \mu} ) \hskip 3mm \text{ for each } \hskip 2mm i \in \mathbb Z$$
by the Frobenius-Nakayama reciprocity. Applying Corollary \ref{asymp}, the first terms of the minimal projective resolution of $K ^{ex} _{( \lambda ^{(0)}, \emptyset )} \boxtimes L _{( \emptyset, \lambda ^{(1)} )}$ (obtained from the double complex arising from the minimal projective resolutions of $K ^{ex} _{( \lambda ^{(0)}, \emptyset )}$ and $L _{( \emptyset, \lambda ^{(1)} )}$) goes as:
\begin{align*}
\cdots \to & \bigoplus _{\gamma, d' > 0} P _{( \gamma, \emptyset )} \left< d' \right> \boxtimes \left( \mathfrak h _1 \otimes P _{( \emptyset, \lambda ^{(1)} )} \left< 2 \right> \right) \oplus \\
& \left( P _{( \lambda ^{(0)}, \emptyset)} \boxtimes \wedge_+ ^2 \mathfrak h _1 \otimes P _{( \emptyset, \lambda ^{(1)} )} \left< 4 \right> \right) \oplus \bigoplus _{\mu, d > 0} \left( P _{( \mu, \emptyset )} \left< d \right> \boxtimes P _{( \emptyset, \lambda ^{(1)} )} \right) \to \\
& \left( P _{( \lambda ^{(0)}, \emptyset )} \boxtimes \mathfrak h _1 \otimes P _{( \emptyset, \lambda ^{(1)} )} \left< 2 \right> \right) \oplus \bigoplus _{\nu, d > 0} \left( P _{( \nu, \emptyset )} \left< d \right> \boxtimes P _{( \emptyset, \lambda^{(1)} )}  \right) \to \\
& P _{( \lambda ^{(0)}, \emptyset )} \boxtimes P _{( \emptyset, \lambda ^{(1)})} \to K ^{ex} _{( \lambda ^{(0)}, \emptyset )} \boxtimes L _{( \emptyset, \lambda ^{(1)} )} \to 0,
\end{align*}
where $\gamma,\mu,\nu$ run over some sets of partitions of $|\lambda^{(0)}|$. We have
$$L _{\bm \mu} = \bigoplus _{w \in \mathfrak S _n / \mathfrak S _{|\mu^{(0)}|} \times \mathfrak S _{|\mu^{(1)}|}} w \cdot L _{( \mu ^{(0)}, \emptyset )} \boxtimes L _{( \emptyset, \mu ^{(1)})}$$
by Fact \ref{typeCref} 1). By examining the $S_{\Gamma}$-action, we conclude that
\begin{align*}
& \mathrm{hom} _{A_W} ( K _{\bm \lambda}, L _{\bm \mu} ) \neq \{0\} \hskip 2mm \text{ only if } \hskip 2mm |\lambda ^{(1)}| = | \mu ^{(1)}|, \text{ and }\\
& \mathrm{ext} ^{i} _{A_W} ( K _{\bm \lambda}, L _{\bm \mu} ) \neq \{0\} \hskip 2mm \text{ only if } \hskip 2mm | \lambda ^{(1)}| - i \le | \mu ^{(1)}| \le | \lambda ^{(1)}|.
\end{align*}
In addition, if  $|\lambda ^{(1)}| = | \mu ^{(1)}|$, then we have
$$\mathrm{ext} ^{\bullet} _{A_W} ( K _{\bm \lambda}, L _{\bm \mu} ) \neq \{0\} \hskip 2mm \text{ only if } \hskip 2mm \lambda ^{(0)} \ge \mu ^{(0)} \text{ and } \lambda ^{(1)} = \mu ^{(1)}$$
by the second part of Lemma \ref{vcc}. Therefore, we conclude that
\begin{equation}
\mathrm{ext} ^{\bullet} _{A_W} ( K _{\bm \lambda}, L _{\bm \mu} ) = \{0\} \hskip 3mm \text{ if } \hskip 3mm a ( {\bm \lambda} ) \ge a ( {\bm \mu} ) \hskip 2mm \text{ and } \hskip 2mm {\bm \lambda} \neq {\bm \mu}.\label{ae}
\end{equation}
By construction, we know that each $K_{\bm \lambda}$ is an indecomposable module with simple head $L_{\bm \lambda}$. Again by counting $S_{\Gamma}$-eigenvalues and using Fact \ref{typeCref} 1), we deduce
$$[ K_{( \lambda^{(0)}, \lambda ^{(1)})} : L _{(\mu^{(0)},\mu^{(1)})}] \neq 0 \hskip 1mm \text{ only if } \hskip 1mm  | \lambda ^{(0)} | > | \mu ^{(0)} |, \hskip 1mm \text{ or } \hskip 1mm \lambda ^{(0)} \le \mu ^{(0)} \text{ and }\lambda^{(1)} = \mu^{(1)}.$$
Hence, Lemma \ref{aasymp} and (\ref{ae}) imply that $K_{\bm \lambda}$ is a $\mathcal P$-trace with respect to $Z^{r,s}_n$. Applying Proposition \ref{orthex}, we conclude that $\{ K _{\bm \lambda} \} _{\bm \lambda}$ forms a Kostka system adapted to $Z ^{r,s} _n$ as required.

\begin{remark}
The $\mathrm{ext}^1$ and gEP-version of the second part of Lemma \ref{vcc} follows by Theorem \ref{kosA}. This yields $\mathrm{ext}^{1} ( K _{\bm \lambda}, L _{\bm \mu} ) = \{ 0 \}$ and $\left< K _{\bm \lambda}, L _{\bm \mu}\right> _{\mathsf{gEP}} = 0$ in place of (\ref{ae}), and hence one can make the proof into a purely algebraic one.
\end{remark}

\end{document}